\documentclass[12pt,leqno,twoside]{amsart}

\usepackage[latin1]{inputenc}
\usepackage[pdftex]{graphicx}
\usepackage[T1]{fontenc}
\usepackage[colorlinks=true, pdfstartview=FitV, linkcolor=blue, citecolor=blue, urlcolor=blue]{hyperref}
\usepackage{amstext,amsmath,amscd, bezier,indentfirst,amsthm,amsgen,enumerate, geometry}
\usepackage[all,knot,arc,import,poly]{xy}
\usepackage{amsfonts,color}
\usepackage{amssymb}
\usepackage{epsfig}
\usepackage{graphicx}

\theoremstyle{plain}
\newtheorem{theorem}{Theorem}
\newtheorem{lemma}{Lemma}
\newtheorem{proposition}{Proposition}
\newtheorem{corollary}{Corollary}

\newtheorem{definition}{Definition}

\newtheorem{example}{ Example}

\numberwithin{equation}{section}
 \numberwithin{theorem}{section}
\numberwithin{proposition}{section}
\numberwithin{lemma}{section}
\numberwithin{corollary}{section}
\numberwithin{definition}{section}
 \numberwithin{remark}{section}
\numberwithin{example}{section}
\numberwithin{figure}{section}

\def\C{{\mathbb C}}
\def\J{{\mathbb J}}
\def\P{{\mathbb P}}
\def\R{{\mathbb R}}
\def\Z{{\mathbb Z}}

\def\bL{{\mathbb L}}

\def\Bi{{\mathcal B}}

\newcommand{\Sing}{{\rm Sing}}

\author{ Susumu TANAB\'E, Abuzer G\"UND\"UZ}
\title[Curves approaching asymptotic critical value set ]{On  curves  approaching  asymptotic critical value set of a polynomial map }

\begin{document}
\maketitle

\begin{center}

\begin{minipage}[t]{12.2cm}



\vspace{0.5cm}

{\sc Abstract -} {\em We present an  effective method to construct curves  approaching the asymptotic critical value set of a polynomial map.  For this purpose, we propose a way to construct rational curves with parametric representation.  In this manner, we show that the asymptotic critical value set contains the critical value of an polynomial associated to so called bad face of the Newton polyhedron. In the case where the polynomial map is non-degenerate at infinity, we also give a superset of the asymptotic critical value set.
 Our main technical tool is the toric geometry that has been introduced into the study of this question by A.N\'emethi and A.Zaharia.
 }

\end{minipage}
\end{center}

\vspace{1.0cm}

\section{Introduction
\label{sec:introduction}}
The bifurcation locus of a polynomial map $f\colon \C^n \to \C$  is the smallest subset $\Bi(f) \subset \C$ such that $f$ is a locally trivial $C^\infty$-fibration over $\C\setminus \Bi(f)$.
It is known that $\Bi(f)$ is the union of  the set of critical values $f(\Sing f)$
  and the set of bifurcation values at infinity $\Bi_\infty(f)$ which may be non-empty and disjoint from $f(\Sing f)$ even in very simple examples. Finding the bifurcation locus in the cases $n>2$ is a difficult task and it still remains to be an unreached ideal. Nevertheless,  one can obtain approximations by supersets of $\Bi(f)$ from exploiting asymptotical regularity conditions.


\vspace{1.0cm}
\footnoterule
{
\footnotesize{ Key words: Newton polyhedron, regularity at infinity, critical values}

\footnotesize{AMS Subject Classification: 14Q20, 58K05, 32S05.}

{\footnotesize Acknolwedgements: ST  has been partially supported by Max Planck Institut f\"ur Mathematik, Centre National de la Recherche Scientifique-The Scientific and Technological Research Council of Turkey  bilateral project 113F007 "Topologie des singularit\'es de la surface complexe," Universit\'e Lille 1.  ST and AG are partially supported by the Scientific and Technological Research Council of Turkey 1001 Grant No. 116F130 "Period integrals associated to algebraic varieties."}}

\newpage
   Jelonek and Kurdyka ~\cite{JK0, JK} established  an algorithm for finding the set of asymptotic critical values ${\mathcal K}_\infty (f).$  It is known that ${\mathcal K}_\infty (f)$ is finite and includes $\Bi_\infty(f)$ in this case. Under the condition that the projective closure of the  generic fibre of $f$ in $\P^n$ has only isolated singularities, Parusi\'nski \cite{Parus} proved that $\Bi(f) = {\mathcal K}_\infty (f) \cup f(\Sing f).$
A precedent work  \cite{DTT} established a method to detect the bifurcation set in an efficient way.
It gave an answer to a question raised in \cite{JK} and  \cite{DT}
about the detection of  the bifurcation locus by rational curve with parametric representation.

\vspace{0.5pc}

More concretely, for a real polynomial  $f: \R^n \to \R$ of degree $\leq d$,
authors of  \cite{DTT} consider a real  rational curve $X(t)$,  $lim_{t \rightarrow 0} \| X(t)\| \rightarrow \infty$,  with  parametric representation with length
$ (d+1) d^{n-1}+1$ to attain the asymptotic critical value $lim_{t \rightarrow 0} f(X(t)) \in {\mathcal K}_\infty (f)$.

In our present note,  we propose a method to construct a rational curve with drastically reduced number of terms  present in
its parametric representation (Theorem \ref{theorem1}). Futher in this article,  we shall use the terminology  "parametric length of  a curve"
  to denote this number.
Thus, in our Example \ref{example41}, the parametric length of  a real rational curve has been reduced to $4$ in comparison with $3601$ proposed in
\cite{DTT}. 

Starting from Lemma \ref{lemma:transposeM}, we take into account the condition $(\mu)$  on the vector $q$ \eqref{Gamma} that is always satisfied for a proper choice of the toric data $W$ (\ref{WM}) (Lemma 	\ref{lemma:musaisfied}).  We follow \cite{NZ, Z96}
as for the use of toric geometry in the investigation of the asymptotically non-regular values of $f.$
Our main Theorem \ref{theorem1} states the inclusion  into ${\mathcal K}_\infty (f)$ of  critical values of certain polynomial $f^W_\gamma$, with possibly non-isolated singularities, constructed on a "bad face" $\gamma$ of the Newton polyhedron of $f$
under the condition $(\mu).$ Thus, Corollary \ref{cor:inclusion} establishes an inclusion relation
\begin{equation}\label{def11}
\bigcup_{\gamma:{ bad\; face}} f_\gamma ( {\rm Sing}\; f_\gamma  \cap (\C^\ast)^{dim\;\gamma})\subset  {\mathcal K}_\infty (f)
\end{equation}
that is valid even in the case of non-isolated singularities.

For the case where $f$ is non-degenerate at infinity
 this gives an approximation of ${\mathcal K}_\infty (f)$ formulated in Corollary \ref{cor:inclusionupper } that determines a superset of ${\mathcal K}_\infty (f)$
\begin{equation}\label{def12}
{\mathcal K}_\infty (f) \subset  \bigcup_{\gamma: bad\; face} f_\gamma ( {\rm Sing}\; f_\gamma  \cap (\C^\ast)^{dim\;\gamma})\cup \{0\}.
\end{equation}
In this case, our Corollary \ref{cor:Kinftynumber} gives a refined upper bound estimate of the cardinality  $\# {\mathcal K}_\infty (f)$ in terms of volumes of polyhedra explicitly obtained from  bad faces.
We remark that this estimation
gives a better approximation than \cite[Theorem 2.2, 2.3]{JK0}
under conditions imposed in
Corollary \ref{cor:Kinftynumber} if $dim \;\gamma < n-1.$

In \cite{Ta}, it is shown that the left hand side set of  the relation \eqref{def11}  is contained in the bifurcation set $\Bi(f) $
for $\gamma$ relatively simple bad face (see Definition \ref{def:nonrelativelysimple}) and  $f^W_\gamma$ with isolated singularities on $(\C^\ast)^{dim\;\gamma}.$
As it is known $\Bi(f) \subset {\mathcal K}_\infty (f)$ from \cite{JK0}, our Corollary \ref{cor:inclusionupper }
 represents a new result only for
$\gamma$ non-relatively simple bad face,  if the condition of the isolated singularities at infinity is assumed.

In Section \ref{sec:nonrelativelysimple},  we examine an example of  a polynomial in 5 variables with non-relatively simple bad face.
Even in this situation, we can construct a curve approaching an asymptotic critical value of $f.$
This gives an example to Corollary \ref{cor:inclusionupper }
  that is not covered by \cite{Ta}.
As \cite{Ta} imposes the condition of isolated singularity at infinity, it does not concern our Example \ref{example41} treating the non-isolated singularities.

It is worthy noticing that M.Ishikawa \cite{Ishikawa} established a precise description of $\Bi(f)$ analogous to  \cite[Proposition 6]{NZ} for any polynomial map in two variables, i.e. possibly with non-isolated singularities at infinity.

Our method heavily relies on various kinds of Newton polyhedra  constructed in two different chart systems. The core technique is explained in Proposition
\ref{lemma:C} where the key data like the integer vector $q \in \Z^n$ and the integer $\rho >0$ are introduced. The vector $q \in \Z^n$ is used to calculate the number
$L_0$ (\ref{L0}) that determines the parametric length together with $\rho$  (\ref{rho}).


The first author expresses gratefulness to Mihai Tib{\u{a}}r for having drawn his attention to the question of asymptotic critical values of a polynomial map and for useful discussions. He  thanks Kiyoshi Takeuchi for comments and remarks.

\section{Approach with unimodular subdivision of the dual cone
\label{sec:unimodular}}

To fix notations and fundamental notions, we follow \cite{NZ, Z96}.

Let us consider a polynomial

\begin{equation}	
 f(x) = \sum_{\alpha} a_\alpha x^\alpha
\label{fx}
\end{equation}	
with $f(0)=0$ where the multi-index $\alpha$ runs within the set of integer points $ supp(f)  =\{ \alpha \in (\Z_{\geq 0})^n ; a_\alpha \not =0\}.$  We introduce a convex polyhedron of finite volume $\Delta(f)$ defined as the convex hull of $supp(f)$ in $\R^n$ that is assumed to be of the maximal dimension i.e. $dim\; \Delta(f)=n.$
We denote the convex hull of $ supp(f) \cup \{0\}$ in $\R^n$ by $\widetilde{\Gamma}_-(f).$


\begin{definition}		
For $a\in (\R^n)^*$ we denote by $\Delta^a$ a face of $\widetilde{\Gamma}_-(f)$ determined by the condition $\left<a,y\right>\leq\left<a,x\right>$ for every pair $x\in \widetilde{\Gamma}_-(f) $ and $y\in \Delta^a.$
For a face $\gamma\subset \Delta(f)$ of the Newton polyhedron of $f$  (\ref{fx}), we define $f_\gamma(x) = \sum_{\alpha \in \gamma} a_\alpha x^\alpha$
\end{definition}

\begin{definition}		
For a set $\Lambda \subset (\R_{\geq 0})^n$ we denote by $C(\Lambda)=\{tv; t\in \R_{\geq 0}, \;v\in \Lambda\}$ the cone with the base $\Lambda.$
\label{def:cone}
\end{definition}

\begin{definition}		
Let $K$ be an unimodular simplicial subdivision of $(\widetilde{\Gamma}_-(f))^*$
where $(\widetilde{\Gamma}_-(f))^\ast$ is the dual to $\widetilde{\Gamma}_-(f).$
$$(\widetilde{\Gamma}_-(f))^\ast= \{ a\in (\R^n)^*;\; \left<a,x\right>\geq 0, \;\forall x\in \widetilde{\Gamma}_-(f)) \}$$
$$=\{a \in   (\R^n)^*;\; \left<a,x\right>\geq 0, \;\forall x\in C(\widetilde{\Gamma}_-(f))) \}.$$
\end{definition}

\begin{definition}(\cite{NZ})
We call a face $\gamma \subset \Delta(f)$ {\it bad}, if it satisfies the following two properties.

(i) The affine subspace of dimension = $dim\; \gamma$ spanned by $\gamma$ contains the origin.

 (ii) ($\pm$ condition for the bad face) There exists an hyperplane $H \subset \R^n$ such that  $\gamma = H \cap \Delta(f)$
 defined by an equation $\sum_{j=1}^n p_j x_j=0$ be provided with a pair of indices $i \not = j $ satisfying $p_i p_j <0.$
\label{definition:badface}
\end{definition}

If $a_1, \cdots, a_k$ is an unimodular basis of a $k-$ dimensional cone $\sigma \in K$ i.e. $\sigma=\Sigma_{i=1}^k t_i a_i, t_i \geq 0$, we can choose ${m}_1, \cdots, {m}_n  \in \R^n$
a basis of the dual cone $\sigma^\ast=\{x\in \R^n; \left<x,a\right>\geq 0,\forall a\in \sigma \}$
such that  $\left<a_i,{m}_j\right>=\delta_{ij}, i \in [1; k], j \in [1; n]$ where $\delta_{ij}$ is Kronecker Delta.
From here on we shall  use the notation $i \in [r_1; r_2] \Leftrightarrow i \in \{r_1, \cdots,  r_2\} $ for two integers $r_1 <r_2.$\\
We can further extend the basis $a_1, \cdots, a_k$ to an $n-$dimensional basis
$a_1, \cdots, a_n$ with the aid of supplementary vectors $a_{k+1}, \cdots, a_n$ in such a way that $\mid det (a_1, \cdots, a_n)\mid=1.$
We pose $\sigma^*=\{\sum_{i=1}^n\lambda_im_i; \lambda_j\geq0,  j \in [1;k]\}$ and $V_{\sigma^*}=\{ \sum_{j=k+1}^n m_j\lambda_{j}; \lambda_j \in \R , j\in [k+1; n]\}$ $ \subset \sigma^*.$

Assume that $ \gamma$ is a bad face and a $n-$dimensional cone $\sigma \in K \subset (\widetilde{\Gamma}_-(f))^\ast  \subset (\R^n)^* $ satisfying
\begin{equation}	
\gamma \subset \sigma^*=\{x\in\R^n; \left<\alpha,x\right> \geq 0, \forall \alpha \in\sigma \}
\label{sigmastar}
\end{equation}	
 with a basis $(a_1,\ldots,a_k)$ such that
$$\gamma=\{v \in  \Delta(f); \left<a_i,v\right>=0, \;i=1,\ldots,k \}. $$
Such a basis exists by virtue of Definition ~\ref{definition:badface} $(ii).$


\begin{definition}	
Let $\sigma\in K$ be an unimodular simplicial cone with $dim(\sigma)=k.$
   An algebraic torus of dimension $n-k $ associated to the cone $\sigma$ can be defined as
$$\Phi[\sigma]=(\C^*)^n/ \{(t^{b_1},\ldots,t^{b_n}); t\in (\C^*)^k, (b_1,\ldots,b_n)\in \sigma   \}.$$
We also consider a disjoint union of tori given by
$M_{\overline{\sigma}}=\cup_{\sigma^{'}\subset \sigma} \Phi[\sigma^{'}]\cong \C^k \times (\C^*)^{n-k}\;\ni (u_1,\ldots,u_k,u_{k+1},\ldots,u_n)$ with $\overline{\sigma}=\cup_{\sigma^{'}\subset \sigma}\sigma^{'}$ where  $\sigma^{'}$ run over all subcones of $\sigma.$
\label{Phisigma}
\end{definition}

\begin{definition}
For an integer vector $\alpha =(\alpha_1, \cdots, \alpha_n)\in \Z^n,$ we denote by $\alpha'=(\alpha_1, \cdots, \alpha_k) \in \Z^{k}$ and $\alpha''=(\alpha_{k+1}, \cdots, \alpha_n) \in \Z^{n-k}$ its respective components.
In a parallel way, we introduce two sets of variables $u' \in \C^{k}$ (called affine)  and $u'' \in (\C^\ast )^{n-k}$  (called toric), $u=(u',u'')  \in \C^n_k$ where
$$\C^n_k =\C^k\times (\C^\ast)^{n-k} .$$


\label{definition:u}
\end{definition}

We introduce the following unimodular matrices $M$ and $W$ with integer entries (i.e. complementary vectors $a_{k+1}, \cdots, a_n  \in \sigma^\bot$ )associated to a cone $\sigma \in K.$
\begin{equation}
W=({a_1}^T,\ldots,{a_n}^T)=
\left( \begin{array}{c}
w_1  \\
w_2      \\
\vdots     \\
w_n      \\
\end{array} \right),
W^{-1}=M=
\left( \begin{array}{c}
m_1  \\
m_2      \\
\vdots     \\
m_n
\end{array} \right)
\label{WM}
\end{equation}
where $(m_1,\ldots,m_n)$ basis of $\sigma^*$ and $\sigma^*=\Sigma_{i=1}^k \R_{\geq 0} m_i +\Sigma_{j=k+1}^n \R m_j .$
Especially, we shall take the cone $\sigma$ so that
$ \{ m_1, \cdots,  m_n \} \subset (\R^\ast)^n.$ This choice is possible thanks to the conditions of Definition \ref{definition:badface}.
Further we use the following notation also (see Lemma \ref{lemma:transposeM})
\begin{equation}
M^T=
\left( \begin{array}{c}
\mu_1  \\
\mu_2      \\
\vdots     \\
\mu_n
\end{array} \right).
\label{mu}
\end{equation}

Under the change of variables
\begin{equation}
(x_1, \cdots, x_n)  = (u^{w_1}, \cdots, u^{w_n})
\label{xu}
\end{equation}
we consider
\begin{equation}
f^W(u)  =\sum_{\alpha \in supp(f)} a_\alpha u^{\alpha\cdot W}
\label{fWu}
\end{equation}
 where $W$ as in \eqref{WM}.
For $\alpha \in \Z^n,$ we represent $\alpha. W \in \Z^n$  the integer vector with the aid of its components
$$ \alpha . W = (\lambda_1(\alpha), \cdots, \lambda_n(\alpha)).$$

Due to the choice of the basis $a_1, \cdots, a_k$ and the definition of the cone $\sigma^\ast,$
we have the following.
\begin{lemma}	
For general  $v\in \Delta (f)$ that is not necessarily located on the bad face $\gamma,$ we have
$\lambda_1(v),\ldots,\lambda_k(v) \geq 0.$\\
\end{lemma}
Thus, only $\lambda_{k+1}(v),\ldots,\lambda_n(v)$ may be negative for $v\in \Delta(f)$ in general. For $v\in \gamma$ bad face, by the choice of  $\sigma$, we have $\lambda_j(v)=\left<a_j,v\right> \geq 0$ for $\forall \;j\in [1;n],$ $\forall  v \in \gamma.$  Because of \eqref{sigmastar},	 $f_\gamma^W(0,u'')$ is a polynomial in $u''$ variables. In other words,
$\lambda_1(v)=\left<a_1,v\right>=0,\ldots, \lambda_k(v)=\left<a_k,v\right>=0, \lambda_{k+1}(v)=\left<a_{k+1},v\right>\geq0 ,\ldots, \lambda_n(v)=\left<a_n,v\right>\geq0, \;\forall v\in \gamma.$

The expression $ f^W(u)  $ is a Laurent polynomial with possibly negative power exponents in toric variables $ u''$,  but being restricted to affine variables $u'$,  it gives a polynomial in $u'=(u_1,\ldots,u_k).$
We denote
$$ \vartheta_{u} f^W(u)= (\vartheta_{u_1} f^W(u), \cdots ,  \vartheta_{u_n} f^W(u)), $$
with $  \vartheta_{u_j} = u_j \frac{\partial}{\partial u_j},$ $ j \in [1;n].$
For a critical point $u^\ast= (0, u''_\ast) \in{\C^n}_k$ such that $  \vartheta_{u} f_\gamma^W(u^\ast) =0$,
we introduce the notation $u' = (u_1, \cdots, u_k),$ $U''=(U_{k+1}, \cdots, U_{n})=(u_{k+1}-u_{k+1}^\ast, \cdots, u_{n}-u_{n}^\ast)$ and consider a local expansion of the Laurent polynomial
$f^W(u)$ at $u=u^\ast= (0, u''_\ast) \in\C^n_k$ as follows

\begin{equation}
  f^W(u) = \sum_{\beta \in supp_{u^\ast}(f^W)} a^{\ast}_\beta (u-u^*)^\beta =  \sum_{\beta \in supp_{u^\ast}(f^W)} a^{\ast}_\beta u'^{\beta'}U''^{\beta''},
\label{fWuU}
\end{equation}
for $ supp_{u^\ast}(f^W) := \{ \beta \in \Z^n; a^{\ast}_\beta \not =0 \}.$
Here the expression corresponding to the term $\alpha.W \in (\Z_{\geq 0}^k \setminus \{0\}) \times \Z_{< 0}^{n-k}$
in (\ref{fWu}) shall produce a series in (\ref{fWuU}) with  $ ( {\beta'},{\beta''}) \in (\Z_{\geq 0}^k \setminus \{0\}) \times (\Z_{\geq 0})^{n-k}$ according to the rule
\begin{equation}
 \frac{1}{u_j} = \frac{1}{u_j^\ast} \sum_{\ell \geq 0} (- \frac{U_j}{u_j^\ast})^\ell.
\label{uU}
\end{equation}

\begin{definition}	We consider a  convex polyhedron $\Delta_{u^\ast} (f^W)= $ convex hull of  $ supp_{u^\ast}(f^W)$
that is a $n-$dimensional polyhedron due to the condition $dim\;\Delta(f)=n.$
For every facet (= $n-1$ dimensional face) $\Gamma$ of   $\Delta_{u^\ast}(f^W),$ we can find an integer vector
$q\in \Z^n$ and an integer $r$ such that $ \left<q,{\alpha}\right> \geq r $, $\forall {\alpha} \in \Delta_{u^\ast}(f^W)$,
and $\left<q,{\beta}\right> = r $, $\forall {\beta} \in \Gamma $.
 Here the components of $q\in \Z^n$  can be chosen coprime.
\label{definition:q}
\end{definition}

For the above  mentioned cone $\sigma,$ consider the decomposition.\\
\begin{equation}
f^W(u)=\tilde f^W(u) + R(u)
\label{ftilde}
\end{equation}
with $$\tilde f^W(u)= \sum_{\alpha \in supp(f) \cap (\Z_{\geq 0})^n M }a_\alpha u^{\alpha W}$$ while $R(u)$ corresponds to terms with exponent vectors $\alpha W \notin (\Z_{\geq 0})^n$ such that  negative powers appear, i.e. some of $\lambda_{k+1}(\alpha),\ldots, \lambda_n(\alpha)$ are strictly  negative and $\lambda_1(\alpha),\ldots, \lambda_k(\alpha)\geq 0.$
We shall note that if some of $\lambda_{k+j} (\alpha) $ is strictly negative for $\alpha \in \Delta(f),$ then  $\lambda_i(\alpha)$ for some $i \in [1;k]$ must be strictly positive. If all  $\lambda_i(\alpha)=0$ for all $i \in [1;k]$ (we remark that for every
$\alpha \in \Delta(f)$,  $\lambda_i(\alpha)\geq 0$ for all $i \in [1;n]$), it means that $\alpha$ is located on the bad face $\gamma,$ thus $\lambda_{k+j} (\alpha) \geq 0$ for all  $j \in [1; n-k].$
In other words,

\begin{lemma}
The Laurent polynomial $f^W(0, u'') = f^W_\gamma(u) =
\sum_{\alpha \in \gamma} a_\alpha u^{\alpha . W}$ is a polynomial (with positive power terms)  in $u''$ variables.
\label{lemma:polynomial}
\end{lemma}

For the bad face with $codim\; \gamma = k$ in $\Delta(f),$ one can find at least $k$ different points in $(\sigma^*  \setminus V_{\sigma^*}) \cap supp(f)$ and  thus one can choose at least $k$ linearly independent  points  $ \alpha = \sum_{i=1}^k \lambda_i m_i \in (\sigma^*  \setminus V_{\sigma^*}) \cap supp(f)$  with $\lambda_i \geq 0$ for $i \in [1;k]$ that satisfies  $\alpha.W \in (\Z_{\geq 0}^k \setminus \{0\}) \times \Z^{n-k}.$  There may be contributions from the expansion at $u=u^\ast$ of the rational function $R(u)$ obtained after the principle (\ref{uU}),  but this situation does not influence on the existence of $k$   linearly independent points in  $supp_{u^\ast}(f^W) \cap \R^k$. The fact that $f_\gamma^W(0,u'')   =
\sum_{\alpha \in \gamma} a_\alpha u^{\alpha . W}$ is a polynomial depending on all $u''$ variables on $\C^{n-k}$  and the bad face $\gamma$ contains at least $n-k+1$ points that span a $(n-k)$ dimensional linear subspace of $\R^n$ yields that
\begin{equation}
 dim\; \left (\Delta_{u^\ast}(f^W) \cap \{(0, \alpha''); \alpha''\in (\R_{\geq 0})^{n-k}\} \right) \geq n-k-1.
\label{(3)}
\end{equation}

\section{Curve construction by means of Newton polyhedron
\label{sec:Newtonpoly}}

This section is the core part of this note. First of all we introduce a polyhedron
$\Delta^\ast$ defined as a convex hull of $\cup_{i=1}^n \Delta_{u^\ast}(\left< \mu_i, \vartheta_{u} f^W(u)\right>). $
Here  the polyhedron $\Delta_{u^\ast}(\left< \mu_i, \vartheta_{u} f^W(u)\right>) $
is defined as a convex hull of $supp_{u^\ast}(\left< \mu_i, \vartheta_{u} f^W(u)\right>)$
obtained after the expansion as in (\ref{fWuU}).
\begin{proposition}
Assume that  $\vartheta_{u} f_\gamma^W(u^\ast) =0$ for $ u^\ast =(0, u_\ast'') \in \C^n_k.$
 There is a facet $\Gamma $ of the polyhedron $\Delta^\ast$ satisfying $dim \;(\Gamma \cap \R^{n-k} ) = n-k-1$ defined by a
vector $q \in \Z^n$ such that
\begin{equation}
 \Gamma = \{\beta \in  \Delta^\ast ;  \left<\beta, q \right> \leq  \left<\tilde \beta, q \right> \; \rm{for\; every}\; \tilde \beta \in   \Delta^\ast \}.
\label{Gamma}
\end{equation}
In other words, for any $\beta \in  \Delta^\ast$, the inequality $ \left<\beta, q \right> \leq  \left<\tilde \beta, q \right>$
holds with every $\tilde \beta \in   \Delta_{u^\ast} (\left< \mu_i, \vartheta_{u} f^W(u)\right>),$ $ i \in [1;n].$
We shall further denote by $\rho$ the following integer
\begin{equation}
\rho = min_{\tilde \alpha \in \Delta^\ast} \left<\tilde \alpha,q\right>
\label{rho}
\end{equation}
that is equal to $ \left<\alpha, q\right>$ for $\alpha \in \Gamma.$
\label{lemma:C}
\end{proposition}
 \begin{proof}

(a) First we remark the existence of  a facet $\tilde \Gamma$ of
$\Delta_{u^\ast} ( \left<\mu_j, \vartheta_u \tilde f^W (u)\right>)$ for (\ref{ftilde}) satisfying $dim \;(\tilde \Gamma \cap \R^{n-k} ) = n-k-1$ and $dim \;(\tilde \Gamma \cap \R^{k} ) = k-1$ for a fixed $j.$
This follows from (\ref{(3)}) and the fact that $\alpha.W \in (\Z_{\geq 0}^k \setminus \{0\} )\times \Z^{n-k}$ for at least $k$ linearly independent points satisfying
 $ \alpha \in (\sigma^*  \setminus V_{\sigma^*}) \cap supp(f),$ as it has been noticed just after Lemma \ref{lemma:polynomial}.

(b) For a polynomial with support located in $\Delta_{u^\ast}(f^W),$ we calculate
\begin{equation} \left<\mu_j, \vartheta_u f^W(u)\right>= \sum_{\beta} a^\ast_\beta  \left<\mu_j, \vartheta_u\right> \left (u'^{\beta'} U''^{\beta''} \right)
\label{muthetaf}
\end{equation}
 with
$$ \left<\mu_j, \vartheta_u\right>\left( u'^{\beta'} U''^{\beta''} \right)= \left( \left<\mu_j , (\beta', \beta'')\right>  u'^{\beta'} U''^{\beta''} + \sum_{\ell =k+1}^n \frac{\mu_{j,\ell} \beta_\ell u_\ell^\ast}{U_\ell} u'^{\beta'} U''^{\beta''}\right).$$

(c) The condition $\mid \beta '' \mid \geq 2$ for $\beta'=0 $ in the expression (\ref{muthetaf}) follows from the fact that $\vartheta_u f^W(u^\ast) =0 $ at the point $u^\ast = (0, u''_\ast ) \in \C^n_k.$
The expression (\ref{muthetaf}) shows that $\R^{n-k} \cap \cup_{j=1}^n \Delta_{u^\ast} ( \left<\mu_j, \vartheta_u  f^W (u)\right>)
\subset \R^{n-k} \cap \Delta_{u^\ast} ( f^W (u)).$

(d) Next we see  that for $\beta \in \left((\Z_{\geq 0})^{k} \setminus \{0\} \right) \times (\Z_{\geq 0})^{n-k}$ the convex hull of
$\beta$ and $\Delta_{u^\ast} ( \left<\mu_j, \vartheta\right>\tilde f^W (u))$ has a facet $\tilde \Gamma'$ such that $ \tilde \Gamma' \cap \R^{n-k} = \tilde \Gamma \cap \R^{n-k}.$  This is due to the fact that, if some of $\lambda_{k+j} (\alpha) $ is strictly negative for $\alpha \in \Delta(f),$ then  $\lambda_i(\alpha)$ for some $i \in [1;k]$ must be strictly positive as it has been remarked just before Lemma
\ref{lemma:polynomial}.

(e)
 As it has been shown in (\ref{fWuU}), (\ref{uU}), the exponent of each term present in the expansion
$ R(u)= \sum_{\beta } c_\beta u'^{\beta'}U''^{\beta''}$
satisfies $(\beta',\beta'') \in \left((\Z_{\geq 0})^{k} \setminus \{0\}\right) \times (\Z_{\geq 0})^{n-k}.$
There are only finite number of  power indices $(\beta',\beta'')$ in the convex hull of $\{0\}$ and $\tilde \Gamma$
that may cause correction to the facet $\tilde \Gamma$ as we draw the convex polyhedron $\Delta_{u^\ast}(\left<\mu_j, \vartheta_u f^W\right>).$
After  finitely many repetitive application  of the arguments (c), (d) to $\tilde \Gamma$, $\tilde \Gamma'$ etc.,
we find  a facet (\ref{Gamma})
defined for $q=(q', q'') \in \Z^n$ that satisfies $C(\Gamma) \cap \R^{n-k} \supset C(\tilde \Gamma) \cap \R^{n-k}.$

\end{proof}

See Figure \ref{facetgamma} where the facet $\Gamma$ is illustrated for the Example \ref{example41}.

 We consider the curve
\begin{equation}
Q(t)=(u' (t),u''(t) )=(c't^{q'}+ h.o.t. ,u_*^{''}+c''t^{q''}+ h.o.t. )
\label{(Q)}
\end{equation}
where $q=(q', q'') $ found in  Proposition ~\ref{lemma:C} and $u_*^{''}\neq 0,$ as $  u_*^{''} \in (\C^\ast)^{n-k}.$\\
Here $ c't^{q'} = (c_1't^{q_1'}, \cdots , c_k't^{q_k'} ) $ etc.\\

\begin{definition}	(~\cite{JK0, JK})
Consider  a curve $x= X(t)$ that satisfies the following two conditions
\begin{equation}
lim_{t\rightarrow 0} || X(t)|| = \infty
\label{(I)}
\end{equation}
\begin{equation}
    lim_{t\rightarrow 0} x_i \frac {\partial  f(X(t))} {\partial x_j} \rightarrow 0
\label{(II)}
\end{equation}
for every  pair $ (i, j) \in [1;n]^2.$
We call the value $   lim_{t\rightarrow 0} f(X(t))$
asymptotic critical value of $f$.   We denote by  ${\mathcal K}_\infty(f)$  the set of  asymptotic critical values of $f$.
\label{JK}
\end{definition}	

After ~\cite{DRT}, the image value of $f$ that is not asymptotic critical is called {\it $t-$regular value of $f.$}
If limit $lim_{t \rightarrow 0} f(X(t)) = p_0$ exists for the curve (\ref{(I)}), the negation of the condition (\ref{(II)}) is known as Malgrange condition for the fibre $f^{-1}(p_0)$, i.e. $\exists \epsilon >0$ such that
$$
 lim_{t\rightarrow 0}  || X(t)||  || grad \;f(X(t))|| > \epsilon. $$

To construct a curve $\|X(t)\|\rightarrow \infty$ as above, it is enough to consider only one torus chart $\Phi[\sigma]$ from Definition
\ref{Phisigma}.\\

\begin{lemma}

For $q = (q', q'') \in \Z^n$  found  in Proposition \ref{lemma:C}, the following equivalence holds.
$(i)\;\exists w_i$ such that $\left<(q', 0),w_i\right>\;<0 \Leftrightarrow \;(ii)\;(q',0)\notin \sum_{j=1}^n \R_{\geq 0}\mu_j.$
We call this condition $(\mu).$
\label{lemma:transposeM}
\end{lemma}
 \begin{proof}
$ (i) \Rightarrow (ii).$  We show the contraposition. For the vector $r  =\sum_{j=1}^n t_j\mu_j,$ $t_j \geq 0$ for every
$j \in [1;n],$
$ \left<r, w_i\right> = t_i \geq 0$ for every $i.$

$ (ii) \Rightarrow (i).$
Also by contraposition.  Take $r = \sum_{j=1}^n s_j\mu_j \in \R^n$ such that $ \left<r, w_i\right> =s_i \geq 0$ for every $i.$
As $r= (q', 0) \not = (0,0)$ not every $s_i$ equals to zero, thus $s_j >0$ for some $j.$\\
Compare with \cite[2.3]{Fulton} Claim 1, Claim 2, Exercise.
\label{lemma: negativepower}
\end {proof}
\begin{lemma} The  condition $(\mu)$ of Lemma \ref{lemma:transposeM} is satisfied for properly chosen vectors $a_k, \cdots, a_n$ that form a part of an unimodular basis of $\R^n$ \eqref{WM}.
	\label{lemma:musaisfied}
\end{lemma}
\begin{proof}
	The condition (ii) of Lemma \ref{lemma:transposeM} is satisfied if $m_k, \cdots, m_n \in (\Z_{>0})^{n-k}.$ Vectors $m_k, \cdots, m_n$ belonging to the vector space orthogonal to $a_\ell, \ell \in [1;k],$ satisfy $< m_j, a_i> = \delta_{j,i}, i, j \in [1;k].$ It is clear that the proper choice of the basis of an unimodular lattice $a_i, i \in [k+1;n],$ entails the property $m_i \in (\Z_{>0})^{n-k}, i \in [k+1;n].$   
\end {proof}

\begin{lemma}
 The integer  $\rho$ (\ref{rho})
is strictly positive
for $q$  determined for a facet $\Gamma$  constructed  in Proposition \ref{lemma:C}.
\label{lemma:rho}
\end{lemma}
\begin{proof}
By Definition \ref{definition:q},
$ \forall \beta  \in  \Delta_{u^\ast}(\left<\mu_j, f^W\right> ) $
there exists $\tilde \alpha \in \{ \left<q, \cdot\right> = \rho\}$
such that $\beta = t \tilde \alpha$ for $t \geq 1.$  The number $\rho$ was defined as the minimal value of the linear function $\left<q, \cdot\right>$ on $\Delta^\ast$ and $\left<q, \beta \right> = t \rho \geq \rho$ thus $\rho$ must be strictly positive.

\end{proof}

Let us denote by $X(t)$ the image of the curve $Q(t)$ defined in (\ref{(Q)}) by the map (\ref{xu}).
\begin{lemma}
The condition ($\mu$) of  Lemma  ~\ref{lemma:transposeM} is  sufficient  so that there exist a curve
$\|X(t)\|\rightarrow \infty$ with finite limit $lim_{t \rightarrow 0 }f(X(t)) = lim_{t \rightarrow 0 } f^W(Q(t))$. The equality
$  lim_{t \rightarrow 0 }$  $\vartheta_u f^W (Q(t))$ $ =$ $ 0$ holds and  the limit $ lim_{t \rightarrow 0 } f^W(Q(t))$ corresponds to a critical value of the polynomial $f^W_\gamma(u).$
\label{lemma:singularvalue}
\end{lemma}

\begin{proof}
By  (\ref{(Q)}) and $x_i = u^{w_i}$, we have
$$ x_i(t) = c_i t^{\left<(q',0), w_i\right>} (1 + h.o.t.). $$
The existence of the value $lim_{t \rightarrow 0 }f(X(t))  = lim_{t \rightarrow 0 } f^W(Q(t))$ is clear from the definition of the curve  (\ref{(Q)}).
\end{proof}

 By means of the vectors introduced in Lemma \ref{lemma:transposeM} $(\mu),$ we deduce the following relation
\begin{equation}
  \left( \begin{array}{c}
\vartheta_{x_1} f(x) \\
\vartheta_{x_2} f(x) \\
\vdots     \\
\vartheta_{x_n} f(x)
\end{array} \right)
 = M^T  \left( \begin{array}{c}
\vartheta_{u_1} f^W(u) \\
\vartheta_{u_2} f^W(u) \\
\vdots     \\
\vartheta_{u_n} f^W(u)
\end{array} \right).
\label{MT}
\end{equation}

Let $ \vec\ell =(\ell_1, \cdots, \ell_n) \in (\R^\ast)^n$ be a  vector in general position with non-zero components and denote $ \left< \vec \ell, u^W \right> = \sum_{j=1}^n \ell_j u^{w_j}$.
Then we have

\begin{equation}
 \left<\vec \ell, x\right>   \left( \begin{array}{c}
\partial_{x_1} f(x) \\
\partial_{x_2} f(x) \\
\vdots     \\
\partial_{x_n} f(x)
\end{array} \right)
 =    \left<\vec \ell, u^W \right> \left( \begin{array}{c}
\frac{\mu_1  }{ u^{w_1 }}  \\
\frac{\mu_2  }{ u^{w_2 }}   \\
\vdots     \\
\frac{\mu_n  }{ u^{w_n }}
\end{array} \right)
\left( \begin{array}{c}
\vartheta_{u_1} f^W(u) \\
\vartheta_{u_2} f^W(u) \\
\vdots     \\
\vartheta_{u_n} f^W(u)
\end{array} \right).
\label{lxdf}
\end{equation}

From this equality we  see  that it is enough to look for a curve $Q(t)$ given by  (\ref{(Q)}) such that
\begin{equation}
 \min_{i \not = j} \left<  (q',0), w_i - w_j  \right > + ord \left(  \left<\mu_j, \vartheta_u f^W\right> (Q(t))  \right) >0
\label{(O)}
\end{equation}
for every $   j \in [1;n]$ so that to ensure the condition   (\ref{(II)}).
In fact, a linear combination of LHS of (\ref{lxdf}) for various vectors $\vec \ell$
will produce all $n \times n$ functions present in  (\ref{(II)}).

We define also
\begin{equation}
L_0= max_{i \not = j} \left<(q',0), w_i-w_j\right>.
\label{L0}
\end{equation}

\begin{definition} We shall use the set of indices $\J \subset [1;n]$ defined by
$$ \J = \{j \in [1;n] ; min_{i \not = j} \left<(q',0), w_i-w_j\right><0\}.$$
\label{definition:J}
\end{definition}
The cardinality of  $\J$  is at most $n-1.$

To formulate the main theorem of this section, we introduce a coordinate system on the (arc) space of  rationally parametrised curves of the form  (\ref{(Q)}),
\begin{equation}
Q(t)=(u' (t),u''(t) )=(c'(0)t^{q'}+  c'(1)t^{q' +1}  + h.o.t. , u_*^{''}+c''(0)t^{q''}+  c''(1)t^{q''+1}  + h.o.t.)
\label{(Q)'}
\end{equation}
where $q=(q', q'') \in \Z^n$ with coprime elements characterised in Proposition ~\ref{lemma:C} and Lemma \ref{lemma:rho}.

Here we take into account finite number of coefficients $ c' (j)= (c_1 (j), \cdots , c_k (j) ) \in \C^k,  $  $ c'' (j)= (c_{k+1} (j), \cdots , c_n (j) ) \in \C^{n-k},$ $j \in \Z_{\geq 0}.$
We denote the space of coefficients $\mathcal C $ in such a way that $ {\bf c}= (c',c'') \in \mathcal C$, $c'= (c'(0),  c'(1),  c'(2),  \cdots ), $
 $c''= (c''(0),  c''(1),  c''(2),  \cdots ). $

The following theorem tells us that every critical value of the polynomial $$f^W_\gamma(u)  = \sum_{\alpha \in \gamma \cap supp(f) } a_\alpha u^{\alpha.W}$$
with $\gamma$ bad face is an asymptotic critical value.
It is worthy noticing that  the singular points of $f_\gamma(x)$ can be non-isolated  and no restriction is assumed on the dimension of the bad face $\gamma$ in question.

\begin{theorem}
Let $f \in \C[x_1, \cdots, x_n]$ be a polynomial whose Newton polyhedron $\Delta(f)$ has maximal dimension $n.$ Assume that $\gamma$ is one of its  bad faces like in Definition \ref{definition:badface}. (i)   We can find a curve $X(t)$ satisfying  (\ref{(I)}),  (\ref{(II)}) of Definition \ref{JK}  such  that $lim_{t \rightarrow 0 } f(X(t))$ equals to a critical value of the polynomial $f^W_\gamma(u).$
(ii) This curve is obtained as a image by the map (\ref{xu}) of a curve $Q(t)$ whose coefficients $\bf c  \in \mathcal C$
satisfy $ (L_0 - \rho +1 ) \mid \J\mid -$tuple of algebraic equations for
$\rho$  \eqref{rho}, $L_0$  (\ref{L0}).
(iii) The curve $Q(t)$ mentioned in (ii) has a parametric representation (\ref{(Q)'}) of  parametric length $L_0-\rho+2$,  i.e. we can assume its parametrisation coefficients  $(c' (j),  c'' (j)) =0$ for $j >  L_0-\rho+1.$

\label{theorem1}

\end{theorem}

\begin{proof}
 By Lemma 	\ref{lemma:musaisfied} and Lemma \ref{lemma:singularvalue} we have already shown that the curve under question satisfies  (\ref{(I)}) of Definition \ref{JK}.

Now we need to show that there is a curve  (\ref{(Q)'}) satisfying  (\ref{(II)}). For this purpose we look for a curve that makes the inequality (\ref{(O)}) valid.
Lemma  \ref{lemma:rho} tells us that it is enough to verify (\ref{(O)})  for $j \in \J$ as   $ord \left(  \left<\mu_j, \vartheta_u f^W\right> (Q(t))  \right) \geq \rho >0.$

The expansion of
$\left<\mu_j, \vartheta_u f^W\right> (Q(t)), j \in \J$ in $t$ has the following form
$$ g^j_\rho({\bf c}) t^\rho +  g^j_{\rho+1}({\bf c}) t^{\rho+1} + h.o.t. $$
For each $j \in \J,$ the vector with polynomial entries $g^j_\rho({\bf c})$ depends on all $n$ variables $ (c'(0),c''(0)) \in \C^n \subset \mathcal C$ in view of the choice of $q \in \Z^n$ made in  Proposition ~\ref{lemma:C}.

As $\mid \J \mid <n,$ the system of algebraic equations $g^j_\rho({\bf c}) =0, \forall j \in \J$ has non-trivial solutions  in $\C$ while
 $g^j_\rho({\bf c}) $ effectively depends on $ (c'(0),c''(0)).$

The vector with polynomial entries  $ g^j_{\rho+1}({\bf c})$ effectively depends on $ (c'(0),$ $c''(0),$ $c'(1),$  $c''(1))$  $\in$ $ \C^{2n}$ $ \subset$ $\mathcal C$
 thus the system of equations $g^j_{\rho+1}({\bf c}) =0, \forall j \in \J$ has also non-trivial solutions in $\C.$

In this way, we can find non-trivial solutions to  $ (L_0 +1 -\rho) \mid \J\mid -$tuple of algebraic equations
$$  g^j_\rho({\bf c}) =  g^j_{\rho+1} ({\bf c}) =\cdots =   g^j_{L_0}({\bf c}) =0, \forall j \in \J$$
for $L_0$ (\ref{L0}).

To prove this,  it is enough to show that   $g^j_{\rho+\ell} ({\bf c})$ effectively depends on $ (c'(\ell),c''(\ell))$ that are absent in
 $g^j_{\rho+\tilde \ell} ({\bf c})$ for $\tilde \ell \in [0; \ell-1].$

First we remark that  $g^j_{\rho+\ell} ({\bf c})$  is a sum of monomials of the form
\begin{equation}
const. \prod^{n}_{\nu=1}  \prod_{i_\nu \in I_\nu}
 c_\nu(i_\nu)^{m_{i_\nu, \nu}}
\label{Ccoeffs}
\end{equation}
satisfying the following homogeneity condition
\begin{equation}
\rho+\ell =  \sum^{n}_{\nu=1}  \sum_{i_\nu \in I_\nu} (q_\nu + i_\nu)m_{i_\nu, \nu}
\label{homog}
\end{equation}
with $I_\nu \subset [0,\ell], $ $ m_{i_\nu, \nu} \geq 0,  $
 $\forall i_\nu \in I_\nu, \forall \nu \in [1;n].$

By a simple calculation, we see that non vanishing terms of the following form appear in $g^j_{\rho+\ell} ({\bf c}), \ell \geq 1$:
\begin{equation}
const. \left(\prod^{n}_{\nu=1, \nu \not =\kappa }  \prod_{i_\nu \in I_\nu}
 c_\nu(i_\nu)^{m_{i_\nu, \nu}} \right) \left(\prod_{i_\kappa \in I_\kappa \setminus \{\ell \} }
 c_\kappa (i_\kappa)^{m_{i_\kappa, \kappa}} \right) c_\kappa(\ell)
\label{Ckcoeffs}
\end{equation}
for $\kappa \in [1;k]$ and
\begin{equation}
const. \left(\prod^{n}_{\nu=k+1, \nu \not =\kappa }  \prod_{i_\nu \in I_\nu}
 c_\nu(i_\nu)^{m_{i_\nu, \nu}} \right) \left(\prod_{i_\kappa \in I_\kappa \setminus \{\ell \} }
 c_\kappa (i_\kappa)^{m_{i_\kappa, \kappa}} \right) c_\kappa(\ell)
\label{Ckcoeffs2}
\end{equation}
for $\kappa \in [k+1;n].$
Non vanishing of (\ref{Ckcoeffs}) with  $\kappa \in [1;k]$ is due to the presence of a term proportional
to $u'^{\alpha'}(t)U''^{\beta''}(t)$  such that $\left<q, (\alpha',\beta'')\right> = \rho$ in $\left<\mu_j, \vartheta_u f^W(u)\right>.$
That of  (\ref{Ckcoeffs2})  with $\kappa \in [k+1;n]$ is due to the presence of a term proportional
to $U''^{\alpha''}(t)$  such that $\left<q, (0,\alpha'')\right> = \rho$ in $\left<\mu_j, \vartheta_u f^W(u)\right>$ and $u_\ast'' \not =0.$ This can be seen from Proposition \ref{lemma:C}.
These originating monomials $u'^{\alpha'}(t)U''^{\beta''}(t)$,  $U''^{\alpha''}(t)$ are uniquely determined from power exponents
$ \{m_{i_\nu, \nu} \}_{i_\nu \in I_\nu}$ that can be seen from (\ref{Ccoeffs}), (\ref{homog}):
$$  \alpha_\kappa = 1+  \sum_{i_\kappa \in I_\kappa}  m_{i_\kappa, \kappa} \; {\rm for} \; \kappa \in [1;k] \\;\;\; (resp. [k+1; n]),$$
$$   \alpha_\nu = \sum_{i_\nu \in I_\nu} m_{i_\nu, \nu} \; {\rm for} \; \nu \in [1;k] \setminus \{\kappa \}  \;\;\; (resp. [k+1; n  ] \setminus \{\kappa \}).$$

Thus, no cancellation of terms (\ref{Ckcoeffs}), (\ref{Ckcoeffs2}) happens.
As $\left<\mu_j, \vartheta_u f^W(u)\right>$ , $j \in \J$ contains monomials $u'^{\alpha'}(t)U''^{\beta''}(t)$,  $U''^{\alpha''}(t)$  of the above type, the factor $c_\kappa(\ell)$, $\kappa \in [1;n]$ appears in   $g^j_{\rho+\ell} ({\bf c})$  but it does not appear in
$g^j_{\rho+\tilde \ell} ({\bf c})$, $\tilde \ell \in [0; \ell-1]$
because of (\ref{homog}).
\end{proof}



\begin{corollary}
Under assumptions of  Theorem  \ref{theorem1}, the following inclusion holds
\begin{equation}
 \bigcup_{\gamma} f_\gamma ( {\rm Sing}\; f_\gamma  \cap (\C^\ast)^{dim\;\gamma})\subset  {\mathcal K}_\infty (f)
\label{fgammaKinfty}
\end{equation}
where $\gamma$ runs among bad faces of $\Delta(f)$.
\label{cor:inclusion}
\end{corollary}

\begin {proof}
Theorem  \ref{theorem1} tells us
$$f^W_\gamma ( {\rm Sing}\; f^W_\gamma  \cap (\C^\ast)^{dim\; \gamma})\subset  {\mathcal K}_\infty (f).$$
It is enough to show that
$$ f^W_\gamma ({\rm Sing}\; f^W_\gamma  \cap (\C^\ast)^{dim\; \gamma})=  f_\gamma ({\rm Sing}\; f_\gamma  \cap (\C^\ast)^{n})$$
for $f_\gamma(x)  = \sum_{\alpha \in \gamma \cap supp(f) } a_\alpha x^{\alpha}.$

From Lemma \ref{lemma:polynomial},  $f^W_\gamma(u)$ is a polynomial depending effectively on toric variables $u''$ and independent of affine variables $u'$ (the condition (i) of the  Definition \ref{definition:badface} ).
This means that $\vartheta_{u_1} f^W_\gamma(u)= \cdots= \vartheta_{u_k} f^W_\gamma(u)=0. $
Thus, for $ u''_\ast \in {\rm Sing}\; f^W_\gamma  \cap (\C^\ast)^{dim\; \gamma}, $ the vanishing of the logarithmic gradient vector holds: $ \vartheta_u f^W_\gamma(0,u''_\ast ) =0$.  By using the map $u''(x) = (x^{m_{k+1}}, \cdots, x^{m_n}) $ induced by the inverse to (\ref{xu}), we see  $f_\gamma(x)= f^W_\gamma(0,u''(x)).$ Taking the relation (\ref{MT}) into account, we see that this entails  $ \vartheta_{x} f_\gamma(x_\ast) =0$ for $x_\ast  \in (\C^\ast)^{n}$ that satisfies $u''(x_\ast)= u''_\ast .$

Conversely, if $ \vartheta_{x} f_\gamma(x_\ast) =0$ for $x_\ast  \in (\C^\ast)^{n}$, by (\ref{MT}), we see    $ \vartheta_{u} f^W_\gamma(0,u''_\ast)=0$ for $ u''_\ast =u''(x_\ast)$ the image of the map  (\ref{xu}).
\end{proof}

In \cite{CDT} Theorem 1.1, for $f:$ non-degenerate at infinity,  it is stated that
\begin{equation}
 {\mathcal K}_\infty (f)  \subset \{0\} \cup \bigcup_{\Delta} f_\Delta  ( {\rm Sing}\; f_\Delta  \cap (\C^\ast)^{dim\; \Delta}),
\label{CDT}
\end{equation}
  where the union runs over all ''atypical faces'' of $f$ (faces that satisfy our Definition \ref{definition:badface} (ii) ).

Here we remark the following:
\begin{corollary}
For $f:$ non-degenerate at infinity in the sense of  \cite{CDT},
the following inclusion holds
\begin{equation}
{\mathcal K}_\infty (f) \subset  \bigcup_{\gamma} f_\gamma ( {\rm Sing}\; f_\gamma  \cap (\C^\ast)^{dim\;\gamma})\cup \{0\} ,
\label{fgammaKinftynondeg}
\end{equation}
where $\gamma$ runs among bad faces of $\Delta(f).$
\label{cor:inclusionupper }
\end{corollary}

\begin{proof}
For $\Delta$ an atypical face satisfying (ii)  of  Definition \ref{definition:badface}, but not  (i),
we see that
\begin{equation}
 f_\Delta  ( {\rm Sing}\; f_\Delta  \cap (\C^\ast)^{dim\; \Delta}) =\{0\}.
 \label{fDelta}
\end{equation}
In this case, the face $\Delta$ is contained in a $(dim \; \Delta)$ dimensional
linear space that does not pass through the origin. As $f_\Delta$ is a weighted homogeneous polynomial such that
$ f_\Delta = \sum_{i=1}^n w_i {\vartheta_i}f_\Delta$ for a non-zero rational vector $(w_i)_{i=1}^n,$ we have
   \eqref{fDelta}.

   The relations \eqref{CDT} and   \eqref{fDelta} yield \eqref{fgammaKinftynondeg}.
\end{proof}

In combining Corollaries \ref{cor:inclusion}, \ref{cor:inclusionupper },
we determine ${\mathcal K}_\infty (f)$ up to $\{0\}$
under the assumption of Theorem   \ref{theorem1} for $f:$ Newton non-degenerate at infinity.

For $f:$ depending effectively on two variables, M.Ishikawa \cite[Theorem 6.5]{Ishikawa} established a precise description of $\Bi(f)$
where a set essentially larger than the LHS of \eqref{fgammaKinfty} appears. This situation suggests that the superset of
${\mathcal K}_\infty (f)$ can be essentially larger than the RHS of  \eqref{fgammaKinftynondeg}, if $f$ is not non-degenerate at infinity.

Parusi\'nski \cite{Parus} established the equality
$$ {\mathcal K}_\infty (f) \cup f( {\rm Sing}\; f) = \Bi (f)$$
for the case where the projective closure in $\P^n$ of the generic fibre of $f$ has only isolated singularities on the hyperplane at infinity $H_\infty \subset \P^n.$
His main concern was to look at the case for the Newton polyhedron $ \Delta(f_d)$ with full dimension $(= n-1)$
as the polynomial is decomposed  into homogeneous terms
$f(x)=  \sum_{j=0}^d f_j(x) , deg\; f_j=j.$
In this setting, we see
\begin{corollary}
Let $f$ be a polynomial such that
the projective closure in $\P^n$ of the generic fibre of $f$ have only isolated singularities on the hyperplane at infinity $H_\infty \subset \P^n.$
Assume the conditions imposed in Corollary \ref{cor:inclusion}. Then we have
\begin{equation}
 \bigcup_{\gamma} f_\gamma ( {\rm Sing}\; f_\gamma  \cap (\C^\ast)^{dim\;\gamma})\subset  \Bi (f) \subset \bigcup_{\gamma} f_\gamma ( {\rm Sing}\; f_\gamma  \cap (\C^\ast)^{dim\;\gamma}) \cup \{0\}.
\label{fgammaBi}
\end{equation}
\label{cor:parus}
\end{corollary}
Thus, to decide exactly $\Bi (f)$ in this case, it is enough to verify
$\{0\} \in \Bi (f)$ or not.

\vspace{1pc}
Now we consider an affine lattice $\bL_{\gamma}$  (i.e. a principal homogeneous space of a free Abelian group) with rank $dim\;\gamma$
generated by
\begin{equation}
\gamma_\Z \cdot W = \{\alpha.W;  \alpha \in \gamma \cap \Z^n\}
\label{def:gammaZ}
\end{equation}
for $\gamma:$ a bad face and $W$ \eqref{WM}.
We denote by $\bL_{\gamma} \otimes \R$ the real affine space spanned by $\bL_{\gamma}.$
 In $\bL_{\gamma} \otimes \R$, we introduce the volume form $Vol_{\gamma}$
by setting the volume of an elementary simplex with vertices in $\bL_{\gamma}$ equal to 1
(\cite[\S 2 C]{GZKDiscriminant}).
After \cite[Theorem 3A.2]{GZKDiscriminant} the principal $A-$determinant $D_A(f_\gamma^W+a_0)$
of the polynomial
$\sum_{\alpha \in supp(f_\gamma)} a_\alpha u^{\alpha.W} + a_0$
is a homogeneous polynomial of degree $Vol_\gamma( \overline{\gamma_\Z \cdot W} )$
in $(a_\alpha)_{\alpha \in supp f_\gamma}$ and $a_0$
for $ \overline{\gamma_\Z \cdot W}:$  the convex hull of  $\gamma_\Z \cdot W$ and $\{0\}$ in $\bL_{\gamma} \otimes \R.$
Let us define the volume $Vol_\gamma( \overline{\gamma_\Z})$ in an analogous way to $Vol_\gamma( \overline{\gamma_\Z \cdot W})$
in replacing $W$ by the identity matrix $Id_{dim\;\gamma}.$

We remark that $Vol_\gamma( \overline{\gamma_\Z \cdot W})$ equals to $Vol_\gamma( \overline{\gamma_\Z})$ due to the unimodularity of $W.$

Thus, we establish the following evaluation on the cardinality of ${\mathcal K}_\infty (f).$

\begin{corollary}
For $f$ like in Corollary \ref{cor:inclusionupper } (resp. like in \ref{cor:parus}),
the following inequality holds
\begin{equation}
    \# {\mathcal K}_\infty (f) \leq 1+  \sum_{\gamma: bad\; face} Vol_\gamma( \overline{\gamma_\Z}),
\label{fgammaKinftynumber}
\end{equation}
[resp.
\begin{equation}
    \# \Bi (f) \leq 1+  \sum_{\gamma: bad\; face} Vol_\gamma( \overline{\gamma_\Z}). ]
\label{fgammaBinumber}
\end{equation}
\label{cor:Kinftynumber}
\end{corollary}

We remark that the estimation above \eqref{fgammaKinftynumber}
gives a better approximation than \cite[Theorem 2.2, 2.3]{JK0}
under conditions imposed in
Corollary \ref{cor:Kinftynumber} if $dim \;\gamma < n-1.$

\section{Non relatively simple face
\label{sec:nonrelativelysimple}}

In  \cite{Ta}, the notion of relatively simple face has been introduced.

\begin{definition} (\cite[Definition 1.4]{Ta})
A face $\gamma \subset \tilde \Gamma_{-}(f) \cap \Delta(f)$ is called relatively simple,  if $C(\gamma)^\ast$ $\subset$ 
$(\tilde \Gamma_{-}(f))^\ast$ is simplicial or  $dim\;C(\gamma)^\ast \leq 3$ .
\label{def:nonrelativelysimple}
\end{definition}

The main result Theorem 1.6. of \cite{Ta} relies heavily on the notion of relatively simple faces.
It shows that  the set $ \bigcup_{\gamma} f_\gamma  ( {\rm Sing}\; f_\gamma  \cap (\C^\ast)^{dim\; \gamma})$ where $\gamma$ runs relatively simple bad faces   is contained in the bifurcation value set of a polynomial mapping $f$ under the condition of  non-degeneracy and isolated singularities at infinity. We say that $f$ has isolated singularities at inifinity over $ b \in \bigcup_{\gamma} f_\gamma  ( {\rm Sing}\; f_\gamma  \cap (\C^\ast)^{dim\; \gamma})$
if   $(f_\gamma^W)^{-1}(b) \cap (\C^\ast)^{dim\; \gamma} $
has only isolated singular points for every Laurent polynomial $f_\gamma^W$ (\ref{fWu}) constructed on a corresponding bad face $\gamma.$
It is clear that this definition does not depend on the choice of the unimodular matrix $W$ (\ref{WM}) due to the argument in the proof of Corollary \ref{cor:inclusion}.

In this section, we examine an example of  a polynomial in 5 variables with non-relatively simple bad face (see (\ref{4.2}) ).
Even in this situation, we can construct a curve $X(t)$ satisfying \eqref{(I)}, \eqref{(II)} of Definition \ref{JK} approaching the value $f_\gamma  ( {\rm Sing}\; f_\gamma  \cap (\C^\ast)^{dim\; \gamma})$ for $\gamma:$ non-relatively simple bad face.
This gives an example to Corollary \ref{cor:inclusion} that is not covered by \cite{Ta}. 

1)  Let us begin with a simplicial cone in $\R^4$
 generated by three $1$ dimensional cones $C(\bar v_1), C(\bar v_2), C(\bar v_3)$ where $\bar v_1=(3,3,4,2),$ $\bar v_2=(1,3,5,2),$ $\bar v_3=(3,1,4,2), $
 $\bar v_4=(1,1,1,1).$

Each face of the simplicial cone
\begin{equation}
F_{i,j,k} := C\left(\sum_{s_i+s_j + s_k=1} s_i{\bar v_i} + s_j{\bar v_j} + s_k{\bar v_k} \right)
\label{4.1}
\end{equation}
where $\{ i,j,k\} = \{1,2,3,4\} \setminus \{\ell\} $ for $\ell \not = i,j,k,$  is defined as a subset of a plane $ \{ v \in \R^4; \left <A_{i,j,k},v \right>=0 \}.$ The orthogonal vector to each of the faces is given by:
$A_{1,2,3}=(-2, 0, -4, 11),$
$A_{1,3,4}= (-2, 0, 1, 1),$
$A_{1,2,4}= (1, -5, 2, 2),$
$A_{2,3,4}= (1,1,0,-2).$  We choose the direction of the orthogonal vector in such a way that $\left <A_{i,j,k},{\bar v_\ell} \right>>0$
for every quadruple indices $\{i,j,k,\ell \} = \{1,2,3,4\}.$

We shall construct a non-simplicial cone by means of an additional cone $C(\bar v_5)$ that will be built with the aid of the vector $A_{2,3,4}.$ Namely we choose
$\bar v_5 = \bar v_2+ \bar v_3 + \bar v_4-  A_{2,3,4} = (4,4,10,7). $

We shall convince ourselves that the new non-simplicial cone geneated by five  $1$ dimensional cones $C(\bar v_1), C(\bar v_2), C(\bar v_3)$,  $C(\bar v_4), C(\bar v_5)$
is a convex cone with six faces. Here we recall the Definition \ref{def:cone}.  In fact, we calculate the orthogonal vector to each of faces $F_{i,j,k} $ defined in a  manner similar to (\ref{4.1})  for $\{i,j,k\} = \{1,2,3,4,5\} \setminus \{\ell, p\}$ such that  $\{i,j,k, \ell,p\} = \{1,2,3,4,5\}.$
$A_{2,3,5}=(17,29,-24,8),$
$A_{2,4,5}=(2, -1, 1, -2),$
$A_{3,4,5}=(-1,3,2,-4)$ in addition to
$A_{1,2,3},  A_{1,3,4},  A_{1,2,4} $ already known ($F_{2,3,4}$ is not a face of the newly constructed cone any more).    We have again $\left <A_{i,j,k},\bar v_\ell \right>>0, $ $\left <A_{i,j,k},\bar v_p \right>>0$
for every quintuple indices $\{i,j,k, \ell,p\} $ as above and see thus the newly constructed cone is convex.

2) Now we consider a shift of the apex of the cone towards a vector $\bar v_0 \in (\R_{>0})^4 ,$ say $\bar v_0 = (1,2,3,1).$
We denote the face of the shifted cone $B_{i,j,k} =\bar v_0 + F_{i,j,k}$,  $(i,j,k) = (1,2,3), (1,3,4), (1,2,4), (2,3,5), (2,4,5), (3,4,5).$
The face $B_{i,j,k}$ is a subset of a plane
$ \{ v \in \R^4; \left <A_{i,j,k},v \right>=  \left <A_{i,j,k},\bar v_0 \right> \}.$
In "homogenising" the defining equation of a plane containing $B_{i,j,k}$
we get a plane in $\R^5:$
$ H_{i,j,k}=\{ (x,y,z,w,r) \in \R^5; \left <A_{i,j,k},(x,y,z,w) \right>=  \left <A_{i,j,k}, \bar v_0 \right> r \}.$
In this way we get six planes in $\R^5$ passing through the origin and the intersection
$$  {\bar C}=\cap_{\{i,j,k\}} \{ (x,y,z,w,r) \in \R^5; \left <A_{i,j,k},(x,y,z,w) \right>- \left <A_{i,j,k},\bar v_0 \right> r \geq 0 \}  $$
produces a convex cone. By construction, every plane $ H_{i,j,k}$ contains  a $1$ dimensional cone $C( v_0)$ and $\bar C \supset C( v_0)$ where $ v_0=(\bar v_0,1).$

If we  use the choice  done in  1) and $ v_0=(1,2,3,1,1)$,
the vectors $L_{i,j,k}$ orthogonal to the planes $H_{i,j,k}$ are given by
 $ L_{1,2,3} = (-2,0,-4,11,3),  $
 $ L_{1,3,4} = (-2,0,1,1,-2),  $
$ L_{1,2,4} = (1,-5,2,2,1), $
$ L_{2,3,5} = (17,29,-24,8,-11),  $
$ L_{2,4,5} = (2,-1,1,-2,-1),  $
$ L_{3,4,5} = (-1,3,2,-4,-7).  $
We define vectors  $v_i = (\bar v_i, 0), i \in [1,4]$ in $\R^5.$
The vector $L_{i,j,k} \in (\R^5)^\ast$ is orthogonal to  $v_i, v_j, v_k$ in addition to $v_0.$
We shall check that $\left<L_{i,j,k}, v_\ell \right> \geq 0 $ for every $v_\ell,$ $\ell \in [0,5].$
Except 6 triples shown above, this positivity property is not satisfied for
other triples from  $\{1,2,3,4,5\}$.

The following polynomial
\begin{equation}
f=  -3 x^{v_0} +x^{3 v_0}+ x^{v_1+v_0}+ x^{v_2+v_0}+x^{v_3+v_0}+ x^{v_4+v_0}+x^{v_5+v_0}
\label{4.2}
\end{equation}
has a $1-$ dimensional bad face contained in $C(v_0)$ that is not relatively simple in the sense of Definition \ref{def:nonrelativelysimple}.
In fact, the cone $C(v_0)^{\ast} \in (\R^5)^\ast$ in the dual fan $(\widetilde{\Gamma}_-(f))^*$ that corresponds to the cone $C(v_0)$  is a 4 dimensional cone with 6 generators $L_{i,j,k}$ calculated above.
 In the sequel, we shall show the inclusion
\begin{equation}
\{ \pm 2 \}\subset {\mathcal K}_\infty(f) \subset \{0, \pm 2 \}.
\label{4.3}
\end{equation}

3) Now we shall construct a unimodular cone $\sigma \in K$  of the unimodular simplicial subdivision
$(\widetilde{\Gamma}_-(f))^*.$
For example, if we choose $ a_1= \frac{1}{5}(L_{1,3,4}+L_{1,2,4}+2 L_{1,2,3} ), a_2= \frac{1}{15}( L_{1,2,4} + 3 L_{1,2,3}+ 10 L_{2,4,5}), a_3=   L_{1,2,3},    a_4 = L_{2,4,5}, a_5= (1,1,-1,1,0)$,
as the column vectors of
$$  W =
 ({a_1}^T, {a_2}^T,{a_3}^T, a_4^T, a_5^T )=
\left( \begin{array}{c}
w_1  \\
w_2      \\
w_3      \\
w_4 \\
w_5
\end{array} \right) =\left(
\begin{array}{ccccc}
 -1 & 1 & -2 & 2 & 1 \\
 -1 & -1 & 0 & -1 & 1 \\
 -1 & 0 & -4 & 1 & -1 \\
 5 & 1 & 11 & -2 & 1 \\
 1 & 0 & 3 & -1 & 0
\end{array}
\right).$$
they are generators of a unimodular cone $\sigma.$
One shall also verify that $\left<a_5, \alpha\right> > 0$ for all $\alpha \in supp(f).$
As for the method to  obtain unimodular simplicial subdivision of a cone  see \cite{Oka}.

In this situation, the polynomial (\ref{4.2}) will have  the following form

$ f^W(u) ={u_1}^{17} {u_2}^7 {u_3}^{29} {u_5}^6+$ ${u_1}^2 {u_2}^4 {u_4}^5{u_5}^3+{u_1}^2 {u_2} {u_3}^5 {u_5}^3+$
${u_1} {u_5}^2+{u_2}^2 {u_4}^3 {u_5}^5+{u_5}^3-3 {u_5}.$

Consider the expansion (\ref{fWuU})  around the singular point $u^\ast =(0,0,0,1)$ where $\vartheta_u f^W(u^\ast) $ $=$ $0.$
Then we see that $supp_{u^\ast}(\left<\mu_j,f^W \right>) \subset \{v; \left<  q,  v  \right > \geq 5 \} $
for $ j \in [1,4]$ and vector $ q=(5, -20, 3, 15, 5).$
The facet $\Gamma \subset supp_{u^\ast}(\left<\mu_j,f^W \right>) $ treated in Proposition \ref{lemma:C}, i.e. $\Gamma \subset \{v; \left<  q,  v  \right> = 5 \}$ can be found as a convex hull of  points $ (0, 0, 0,$ $ 0,$ $ 1), $  $ (0, 2, 0, 3, 0),$  $ (1, 0, 0, 0, 0),$  $ (2, 1, 5, 0, 0),$
$ (2, 4, 0, 5, 0).$

The relation $W.(q',0) = (-1,0,-2,8,-1)$ for $(q',0)$ $=$ $(5, -20, 3,$ $ 15, 0)$
 shows that the condition $(\mu)$ of Lemma \ref{lemma:transposeM}  is satisfied. From this relation, we see that the index set $\J =\{1,2,4,5 \}$ and  $min_{i \not = j} <(q',0), w_i-w_j> =  <(q',0), w_3-w_4> = -10 = - L_0,$
$\rho = min_{\alpha \in \Delta_{u^\ast}(\left<\mu_j,f^W \right>)  } <q,\alpha> =5. $

We consider a curve $Q(t)$ with real coefficients of length $11= L_0+1$ namely
$$ u_1 = \sum_{j=0}^{10} c_1(j)t^{j+5},  u_2 =  \sum_{j=0}^{10} c_2(j)t^{j-20},$$
$$ u_3 =  \sum_{j=0}^{10} c_3(j)t^{j+3},  u_4 =  \sum_{j=0}^{10}  c_4(j)t^{j+15},  u_5 =  1+ \sum_{j=0}^{10} c_5(j)t^{j+5}.$$

The system of equations (that corresponds to the coefficients of $ t^{5}$ terms)
$$g^1 _ 1 (\textbf{c}) =
 g^2 _ 1 (\textbf{c}) =
  g^4 _ 1 (\textbf{c}) = g^5 _ 1 (\textbf{c}) = 0$$ where

$g^1 _ 1 (\textbf{c}) = 4 c_1(0)^2 c_2(0)^4 c_4(0)^5+2 c_1(0)^2 c_2(0) c_3(0)^5+2 c_1(0)+4 c_2(0)^2 c_4(0)^3+6 c_5(0),$

$g^2 _ 1 (\textbf{c}) =3 c_1(0)^2 c_2(0)^4 c_4(0)^5+3 c_1(0)^2 c_2(0) c_3(0)^5+5 c_1(0)+5 c_2(0)^2 c_4(0)^3+12 c_5(0), $

$g^4_ 1 (\textbf{c}) = 3 c_1(0)^2 c_2(0)^4 c_4(0)^5+2 c_1(0)^2 c_2(0) c_3(0)^5+3 c_1(0)+3 c_2(0)^2 c_4(0)^3+6 c_5(0)$

$g^5_ 1 (\textbf{c}) = c_1(0)^2 c_2(0)^4 c_4(0)^5+c_1(0)^2 c_2(0) c_3(0)^5+c_1(0)+c_2(0)^2 c_4(0)^3+6 c_5(0)$

\noindent
admits non-trivial solutions because each equation effectively depends on ${c_j(0)}_{j \in[1;5]}.$    In a similar manner,
the system of equations (that corresponds to the coefficients of $t^{k+1}$ terms)
$$g^1 _ k (\textbf{c}) =
 g^2 _ k (\textbf{c}) =
  g^4 _ k (\textbf{c}) =  g^5 _ k (\textbf{c}) = 0 $$ for $k\in[2,6]$ also admits non- trivial solutions by virtue of Theorem  \ref{theorem1}.

 In this way, we can find non-trivial solutions for a system of 24 algebraic equations $g^j_k (\textbf{c}) =
  0, \textbf{c} \in {\mathcal C},  j = 1, 2, 4,5, k \in[1, 6].$ This means that we can construct a
curve   $Q(t)$ of parametric length 7 sastisfying the condition (\ref{(O)}) $ -10+ ord \left<\mu_j, \vartheta_u f^W(Q(t)) \right> >0 $ for $ j \in \J =\{1,2,4,5\}$.
The image $X(t)$ of the curve $Q(t)$ by the map
$$  x_1 =u^{ (-1, 1, -2, 2, 1)},x_2= u^{(-1, -1, 0, -1, 1)}, x_3=u^{(-1, 0, -4, 1, -1)}, x_4=u^{(5, 1, 11, -2, 1)}, x_5 = u^{(1, 0, 3, -1, 0)}$$
satisfies \eqref{(I)}, \eqref{(II)} of Definition \ref{JK}
and $lim_{t \rightarrow 0} f(X(t)) = -2 \in {\mathcal K}_\infty(f).$  A similar arguments shows $2 \in {\mathcal K}_\infty(f).$
We see that the polynomial (\ref{4.2}) is non-degenerate at infinity in the sense of \cite{CDT} and its Theorem 1.1 can be applied to it.
There is no contribution in the right hand side superset  in (\ref{CDT}) from "atypical faces" except that from the "strongly atypical face" \cite[Definition 3.2]{CDT} corresponding to the bad face of $\Delta(f)$  for (\ref{4.2}).
Thus, we conclude the inclusion relation (\ref{4.3}).

For the polynomial  (\ref{4.2})  the method of \cite[Theorem 3.5.]{DTT} proposes construction of a curve of parametric length $d^4(d+1)+1=3360001$ with $d= 20 = |v_0+v_1|$ satisfying the required properties.

\section{Examples
\label{sec:examples}}

We shall give an example that illustrates our Theorem \ref{theorem1}.

\begin{example} {\rm (Non-isolated singularity on a two dimensional bad face)

Consider a polynomial $ f(x) =x ^{v_1}+ (x^{v_2} - x ^{v_3}+1)^2 + (x^{v_2} - x ^{v_3}+1)^3 + x^{v_4}-2$ with
$v_1 =(2,1,1), v_2= (2,2,1), v_3 = (1,2,1), v_4 = (3,1,1) . $
This case with non-isolated singularities at infinity has not been treated in \cite{Ta}.

We remark that
$$
M= \left( \begin{array}{c}
v_1  \\
v_2      \\
v_3
\end{array} \right)= ({\mu_1}^T, {\mu_2}^T,{\mu_3}^T)=
\left(
\begin{array}{ccc}
 2 & 1 & 1 \\
 2 & 2 & 1 \\
 1 & 2 & 1
\end{array}
\right)$$
is unimodular. Thus, we can set
$$
M^{-1}= W= ({a_1}^T, {a_2}^T,{a_3}^T)=
\left( \begin{array}{c}
w_1  \\
w_2      \\
w_3
\end{array} \right) =
\left( \begin{array}{ccc}
0& 1 & -1 \\
-1& 1  & 0 \\
2& -3  & 2
\end{array} \right).$$

The only  bad face $\gamma$ of $\Delta(f)$ is located on the plane spanned by $v_2, v_3.$
For the above $W,$ we have
$$f^W(u)  =-2+ u_1+ (u_2 - u_3+1)^2 + (u_2 - u_3+1)^3 + \frac{u_1 u_2}{u_3}. $$
After \eqref{fgammaKinftynumber}, $\# {\mathcal K}_\infty (f) \leq 10$
while $  Vol_\gamma( \overline{\gamma_\Z \cdot W}) =9 $ for $\overline{\gamma_\Z \cdot W}:$
the convex hull of $\{(0,0), (3,0), (0,3)\}.$

The polynomial  $f_\gamma^W(0, u_2, u_3)  =(u_2 - u_3+1)^2 + (u_2 - u_3+1)^3 $ has non-isolated singularities along a line
$u_2 - u_3+1=0.$ We can choose, for example $u^\ast = (0,-1/3,2/3).$ In the neighbourhood of this point the rational function $f^W(u)$
has the expansion
$$ f^W(u)  = -2+ u_1+ (U_2 - U_3)^2 + (U_2 - U_3)^3  +\frac{3u_1}{2} (U_2-1/3)(1- \frac{3 U_3}{2} + ( \frac{3 U_3}{2} )^2 + \cdots) , $$
for $U_2 = u_2 +1/3, U_3= u_3 -2/3.$\\
$\Delta( \left<\mu_i, \vartheta_u f^W(u)\right>)\;\;i=1,2,3$ give rise to the facet $\Gamma $, (\ref{Gamma}).

A direct calculation shows
$$\left<\mu_3, \vartheta_u f^W(u)\right> = \frac{ u_1}{16}-2 U_2 + 2 U_3 + h.o.t.$$

\begin{figure}[ht]										
\begin{center}											
\includegraphics[width=8cm]{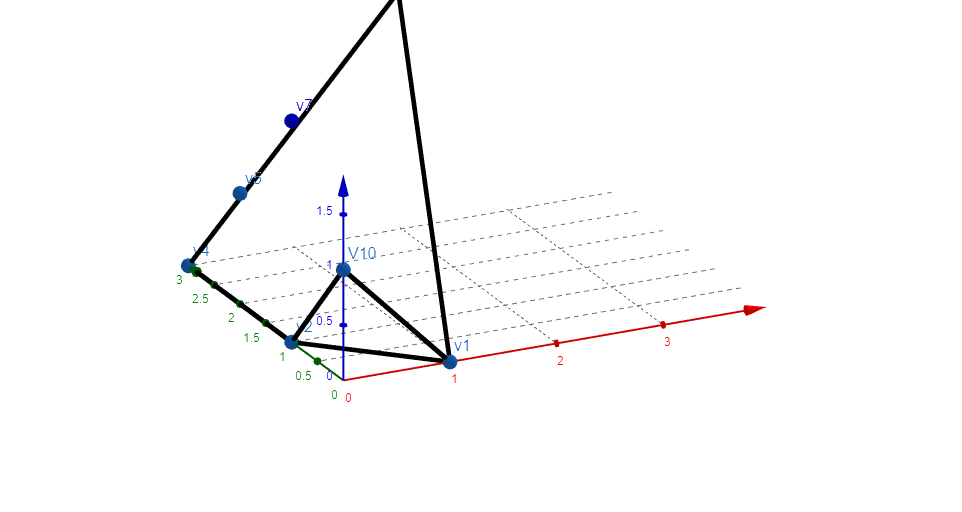} 		
\caption{ The facet $\Gamma$  }
\end{center}			
\label{facetgamma}								
\end{figure}

Thus, we find the facet $\Gamma$ located on the plane containing $(1,0,0), (0,0,1), (0,1,0)$ and $q=(1,1,1),$ $(q',0)=(1,0,0).$

We calculate $L_0 = 3$ and $\rho =1.$ A curve $Q(t)$ (\ref{(Q)'}) with real coefficients of parametric length $4,$ namely

$$ u_1 = \sum_{j=0}^3 c_1(j)t^{j+1}, u_2 = -1/3 + \sum_{j=0}^3 c_2(j)t^{j+1}, u_3 = 2/3 + \sum_{j=0}^3 c_3(j)t^{j+1} $$

that satisfies $$ -3+ord \left<\mu_j, \vartheta_u f^W(Q(t))\right> >0$$ for $j \in \J = \{1, 3\}$ can be constructed.

We remark that after the method of  \cite[Theorem 3.5.]{DTT}, the real curve with required property has parametric length $16\times 15^2 +1=3601 .$

In fact, if we plug these expressions into $\left<\mu_3, \vartheta_u f^W(u)\right>,$ we get an expansion with initial term proportional to $t^1$, ($\left<q, \alpha\right> =1$ for $\alpha \in \Gamma$);

$ \left<\mu_3, \vartheta_u f^W(u)\right> (Q(t)) = \{ c_1(0)/2 - 2 c_2(0) + 2 c_3(0)) \} t +$\\
$ 1/4  \{2 c_1(1) + 6 c_1(0) c_2(0) - 4 c_2(0)^2 - 8 c_2(1) + 3 c_1(0) c_3(0) + 8 c_2(0) c_3(0) -
   4 c_3(0)^2 + 8 c_3(1)\} t^2+$\\
$  1/8 \{ 4 c_1(2) + 12 c_1(1) c_2(0) + 24 c_2(0)^3 + 12 c_1(0) c_2(1) - 16 c_2(0) c_2(1) -
   16 c_2(2) + 6 c_1(1) c_3(0) - 18 c_1(0) c_2(0) c_3(0) - 72 c_2(0)^2 c_3(0) + 16 c_2(1) c_3(0) -
   9 c_1(0) c_3(0)^2 + 72 c_2(0) c_3(0)^2 - 24 c_3(0)^3 + 6 c_1(0) c_3(1) + 16 c_2(0) c_3(1) -
   16 c_3(0) c_3(1) + 16 c_3(2)\}t^3+\cdots$

The  case $\left<\mu_j, \vartheta_u f^W(Q(t))\right>$ also admits a similar expression.
In both cases  $j \in \J = \{1, 3\},$ coefficient of $t$ depends on $(c_1(0), c_2(0), c_3(0)),$ that of $t^2$ depends on $(c_1(0),$ $ c_2(0),$ $c_3(0),$ $c_1(1),$ $c_2(1),$ $ c_3(1)),$ that of $t^3$ depends on $(c_i(j))_{i=1,2,3, j=0,1,2}.$ Thus, the system of algebraic equations imposed on $(c_i(j))_{i=1,2,3, j=0,1,2} \in \C^9$ to make the coefficients of $t, t^2, t^3$ vanish has non-trivial solutions. In fact, this system can be solved in $\R^9.$ As for the construction of a real curve $Q(t),$ i.e. a real curve $X(t),$ see \cite{TGE}.

We can choose as $(c_1(3), c_2(3) , c_3(3)) \in \C^3$ arbitrary non-zero vector.

The image of the curve $Q(t)$ by the map
$$ x_1 = u_2 u_3^{-1}, x_2= u_1^{-1}u_2, x_3 = u_1^2 u_2^{-3} u_3^2 $$
satisfies (\ref{(I)}) ,  (\ref{(II)}) of Definition \ref{JK} and $lim_{t \rightarrow 0} f(X(t)) = -2 \in {\mathcal K}_\infty(f).$

As it can be seen in this example, the curve  $X(t)$ approaches the surface $\{x; f(x) = -2\}$ as $t \rightarrow 0.$

We obtain a curve $X(t) =(x_1(t), x_2(t), x_3(t))$ asymptotically approaching  the surface $\{x ; f(x) =-2\}$
as follows:
$$ x_1(t) = \frac{t^4+t^3+t^2+t-\frac{1}{3}}{t^4+\frac{131 t^3}{256}-\frac{t^2}{4}+\frac{3 t}{4}+\frac{2}{3}}$$
$$x_2(t)=\frac{t^4+t^3+t^2+t-\frac{1}{3}}{t^4+t^3+t^2+t}$$
$$x_3(t)=\frac{\left(t^4+\frac{131 t^3}{256}-\frac{t^2}{4}+\frac{3 t}{4}+\frac{2}{3}\right)^2
   \left(t^4+t^3+t^2+t\right)^2}{\left(t^4+t^3+t^2+t-\frac{1}{3}\right)^3}.$$

\begin{figure}[ht]										
\begin{center}											
\includegraphics[width=8cm]{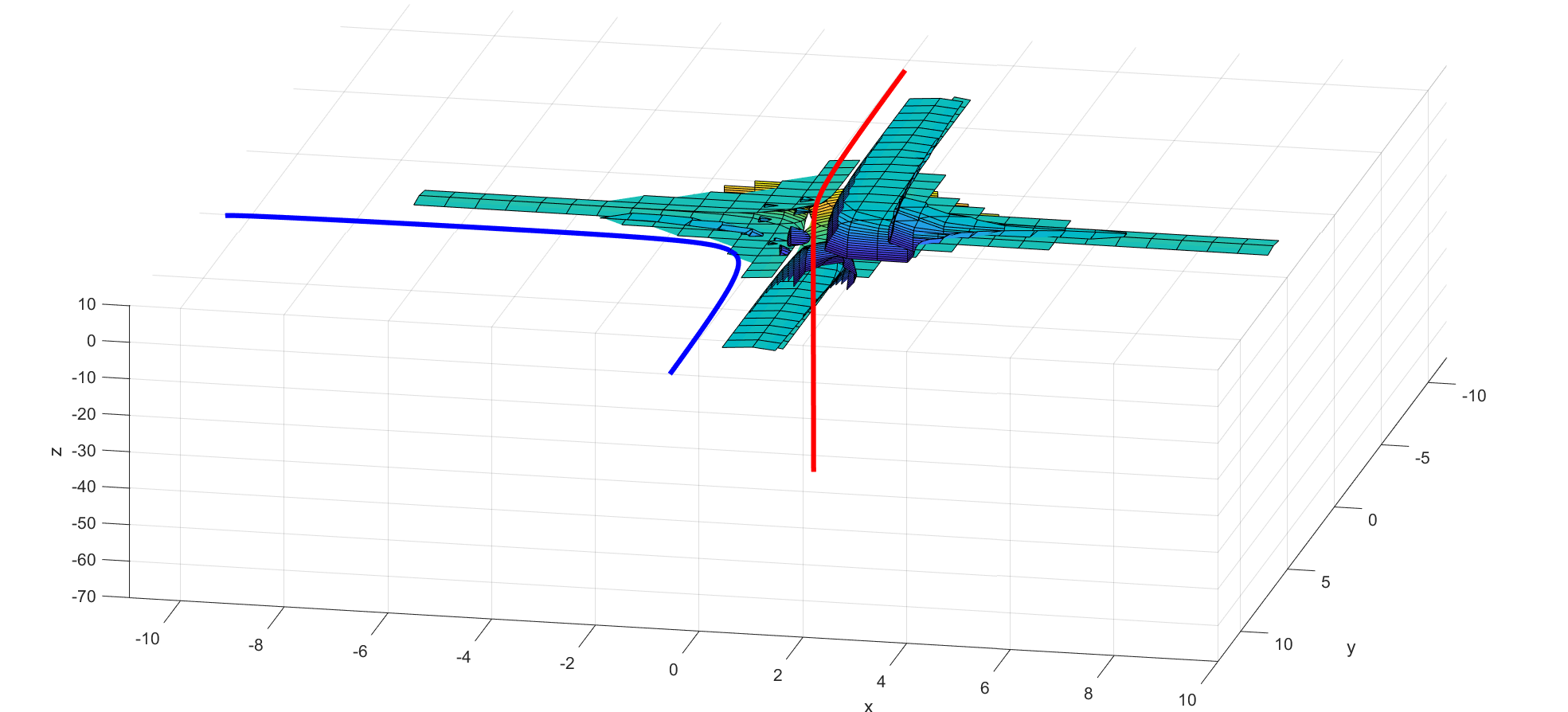}
\caption{ Branches of the curve  $X(t)$  }		
\end{center}		
\label{figure512}										
\end{figure}

 On the Figure 5.2, we see two branches of the curve  that correspond to the asymptotes $ t\rightarrow 0$ from $t>0$ and $t<0$ respectively.
In Example \ref{example41}, the figure illustrating algebraic surface and rational parametric curves are prepared with the aid of the computer programme MATLAB.
}
\label{example41}
\end{example}

\begin{example}
{\rm (Isolated singularities at infinity)
Consider a polynomial $ f(x) =-3 x^{v_0} + x ^{v_1} + x^{v_2} + x ^{3 v_0}$ with
$v_0 =(2,2,1), v_1= (1,0,1), v_2 = (0,1,1). $
$$
W= ({a_1}^T, {a_2}^T,{a_3}^T)=
\left( \begin{array}{c}
w_1  \\
w_2      \\
w_3
\end{array} \right) =
\left( \begin{array}{ccc}
1& 0 & 1 \\
-2& -1  & -1 \\
2& 2  & 1
\end{array} \right).$$

$$
M= \left( \begin{array}{c}
m_1  \\
m_2      \\
m_3
\end{array} \right)= ({\mu_1}^T, {\mu_2}^T,{\mu_3}^T)=
\left(
\begin{array}{ccc}
 -1 & -2 & -1 \\
 0 & 1 & 1 \\
 2 & 2 & 1
\end{array}
\right).$$
The only  bad face $\gamma$ of $\Delta(f)$ is located on the cone $\{t.v_0; t >0 \}.$
$$f^W = {u_1}^3 {u_2}^2 {u_3}^2+{u_2}+{u_3}^3-3 {u_3}.$$
and  $f^W_\gamma (u)  = {u_3}^3-3 {u_3} $ has singular points $u_3^\ast = \pm 1$ and critical values $ \mp 2 $ respectively.
After \cite{Ta} we see that in this case the bifurcation set $\Bi(f) \subset  {\mathcal K}_\infty (f)$ contains $\{\pm 2\}.$

We shall construct a curve $X(t)$ that satisfies (\ref{(I)}) ,  (\ref{(II)}) of Definition \ref{JK} and also the limit condition $  lim_{t\rightarrow 0} f(X(t)) = -2.$ A curve
satisfying $  lim_{t\rightarrow 0} f(X(t)) = 2$ can be also constructed in a  parallel way.

First of all we calculate $\left<\mu_j, \vartheta_u f^W(u)\right>, j =1,2,3$ and find the facet $\Gamma$ as in Proposition \ref{lemma:C}. For example
 $\left<\mu_3, \vartheta_u f^W(u)\right> $ has the following form with $U_3 = u_3-1,$
$$3 {u_1}^3 {u_2}^2 ({U_3}+1)^2-{u_2}+3 {U_3} ({U_3}+1) ({U_3}+2).$$
The facet $\Gamma$ is on the plane containing $(3,2,0), (0,1,0),(0,0,1)$ and $q = (-1,3,3)$ i.e. $(q',0)= (-1,3,0).$
As we see
$$ \left<(q',0), w_1\right> = -1,  \left<(q',0), w_2\right> = -1,  \left<(q',0), w_3\right> = 4,$$
$\J =\{3\}$ and $L_0= max_{i\not = j} \left<(q',0), w_i- w_j\right> = 5.$
The curve (\ref{(Q)'}) has the expansion
$$ u_1 = c_1(0) t^{-1} +  c_1(1) + h.o.t. , u_2 = c_2(0) t^{3} +  c_2(1)t^4 + h.o.t,  u_3 =1+ c_3(0) t^{3} +  c_3(1)t^4 + h.o.t. $$
If we plug these expressions into $\left<\mu_3, \vartheta_u f^W(u)\right>, $ we get an expansion with initial term $t^3$ ($\left<q, \alpha\right> =3$ for $\alpha \in \Gamma$);
$$\{c_2(0) + c_1(0)^3 c_2(0)^2 + 6 c_3(0)\} t^3+
\{ 3 c_1(0)^2 c_1(1) c_2(0)^2 + c_2(1) + 2 c_1(0)^3 c_2(0) c_2(1) $$ $$ + 6 c_3(1)\} t^4+
\{3 c_1(0) c_1(1)^2 c_2(0)^2 + 3 c_1(0)^2 c_1(2) c_2(0)^2 + 6 c_1(0)^2 c_1(1) c_2(0) c_2(1) +$$
$$+ c_1(0)^3 c_2(1)^2 + c_2(2) + 2 c_1(0)^3 c_2(0) c_2(2) + 6 c_3(2) \} t^5 + h.o.t.$$
The coefficient of $t^3$ depends on $(c_1(0), c_2(0),c_3(0))$ that of $t^4$ depends on  $(c_1(0),$ $ c_2(0),$ $ c_3(0),$ $ c_1(1),$ $ c_2(1),$ $c_3(1))$
that of $t^5$ depends on  $(c_1(0), c_2(0), c_1(1), c_2(1), c_1(2), c_2(2),c_2(3) ).$
Thus,  we can construct a curve such that $-5+ ord  \left<\mu_3, \vartheta_u f^W\right>(Q(t)) > 0.$ The minimum parametric length of such a curve $Q(t)$ (\ref{(Q)'}) is $ 4.$ Here coefficients can be chosen to be real. As for the construction of a real curve $Q(t),$ i.e. a real curve $X(t),$ see \cite{TGE}.
We remark that after the method of \cite[Theorem 3.5.]{DTT}, the rational curve  with required properties has length $3601.$

We get the desired curve $X(t)$ as the image of the curve $Q(t)$ by the map

$$  x_1 =u_1 u_3,  x_2= (u_1^2 u_2 u_3)^{-1},  x_3= u_1^2 u_2^2  u_3.$$
We obtain a curve $X(t) =(x_1(t), x_2(t), x_3(t))$ asymptotically approaching the surface $\{x ; f(x) =-2\}$
as follows:
$$ x_1(t) =\left(t^2+t+\frac{1}{t}+1\right) \left(t^6-\frac{8 t^5}{3}-t^4-\frac{t^3}{3}+1\right)$$
$$x_2(t)=\frac{1}{\left(t^2+t+\frac{1}{t}+1\right)^2 \left(t^6-\frac{8 t^5}{3}-t^4-\frac{t^3}{3}+1\right) \left(t^6+t^5+t^4+t^3\right)} $$
$$x_3(t)=\left(t^2+t+\frac{1}{t}+1\right)^2 \left(t^6-\frac{8 t^5}{3}-t^4-\frac{t^3}{3}+1\right) \left(t^6+t^5+t^4+t^3\right)^2.$$

\begin{figure}[ht]										

\begin{center}											
\includegraphics[width=8cm]{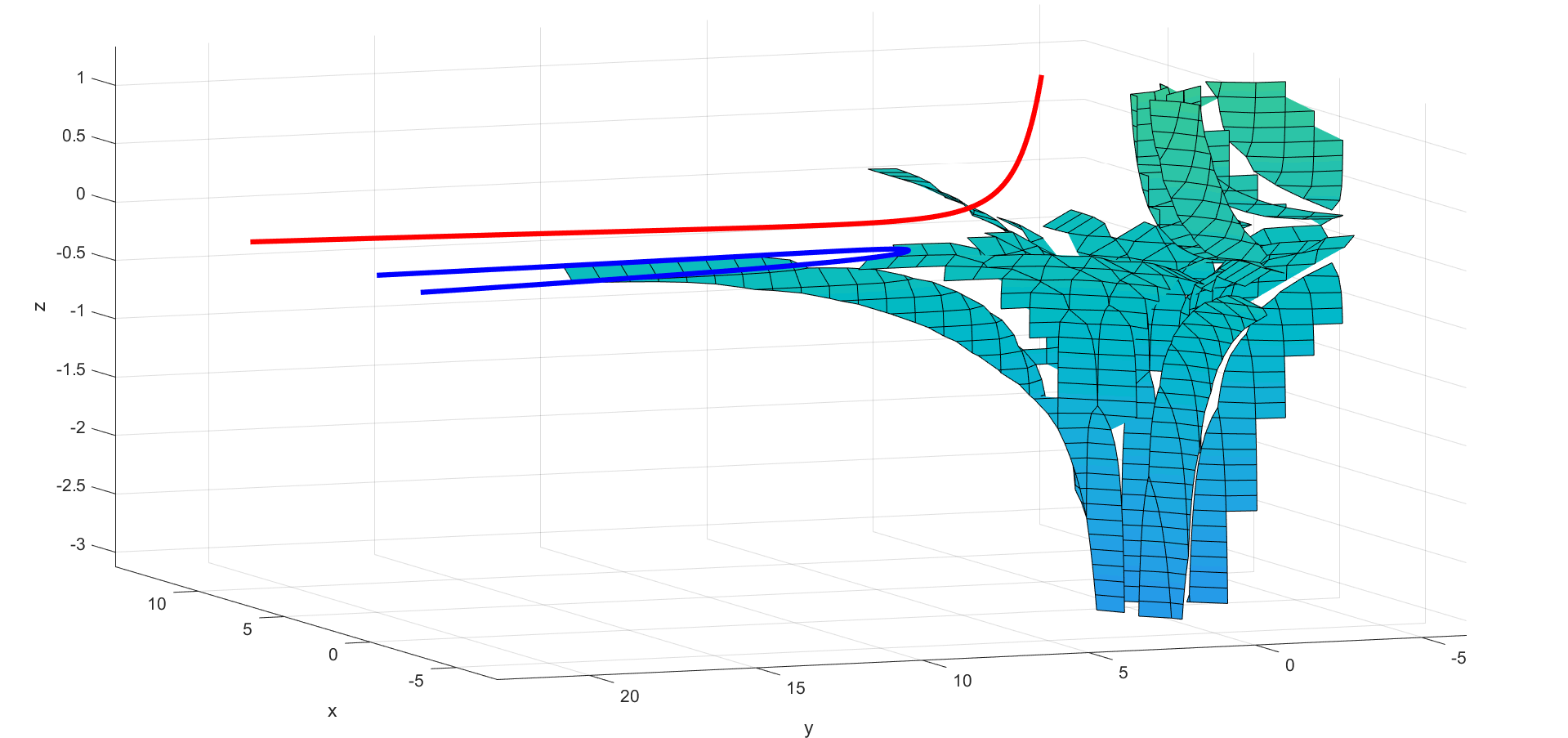} 	
\caption{ Branches of the curve  $X(t)$  }		
\end{center}
\label{figure52}												
\end{figure}
 On the Figure 5.3, we see two branches of the curve  that correspond to the asymptotes $ t\rightarrow 0$ from $t>0$ and $t<0$ respectively.
}
\label{example42}
\end{example}

\begin{example}
{\rm($A_n$ singularity on a $1-$dimensional  bad face )
More generally consider a polynomial $$ f(x) =  \prod_{j=1}^q ( x^{v_0}-z_j)^{m_j} + b_1 x ^{v_1} + b_2 x^{v_2}+ b_3 x^{v_3} - \prod_{j=1}^q ( -z_j)^{m_j}$$
with $b_1 b_2 b_3  \prod_{j=1}^{q}z_j \not = 0,\; m_1\geq 2,$ and
$v_0 =(2,2,1), v_1= (0,1,1), v_2 = (1,0,1),v_3=(1,1,3) $. We can choose $ a_1\bot span \{v_0,v_1\} \cap \Delta(f), a_2\bot span \{v_0,v_2 \}\cap \Delta(f) ,$ $a_3\bot span \{v_1,v_2 \}\cap \Delta(f)$ and get $a_1=(1,-2,2), a_2=(-2,1,2), a_3=(1,1,-1)$.

 We take a simplicial unimodular subdivison  of the cone generated by $a_1, a_2, a_3$ as $a_1^{(1)}=\frac{a_2+2a_1}{3}=(0,-1,2), \;a_1^{(2)}=\frac{a_2+a_1^{(1)}}{2} =(-1,0,2), a_3^{(1)}=(1,-1,1)$ (see \cite{Oka})

$$
W= ({a_1^{(2)}}^T, {a_1^{(1)}}^T,{a_3^{(1)}}^T)=
\left( \begin{array}{c}
w_1  \\
w_2      \\
w_3
\end{array} \right) =
\left( \begin{array}{ccc}
-1& 0 & 1 \\
0& -1  & -1 \\
2& 2  & 1
\end{array} \right)$$

$$
M= \left( \begin{array}{c}
m_1  \\
m_2      \\
m_3
\end{array} \right)= ({\mu_1}^T, {\mu_2}^T,{\mu_3}^T)=
\left(
\begin{array}{ccc}
 1 & 2 & 1 \\
 -2 & -3 &- 1 \\
 2 & 2 & 1
\end{array}
\right).$$

The only  bad face $\gamma$ of $\Delta(f)$ is located on the cone $C(v_0)= \{t.v_0; t >0 \}.$\\
 We assume $z_1\cdot z_2\ldots \cdot z_q\neq0,\;\;z_i\neq z_j,\;\;\forall i,j \in [1:q].$
$$f^W(u) = \prod_{j=1}^{q}(u_3-z_j)^{m_j} +b_1 {u_1}^2 {u_2}+b_2{u_1}{u_2}^2{u_3}^2 +b_3{u_1}^5{u_2}^5{u_3}^3 - \prod_{j=1}^{q}(-z_j)^{m_j}$$
and  $f^W_\gamma (u)  = \prod_{j=1}^{q}(u_3-z_j)^{m_j} - \prod_{j=1}^{q}(-z_j)^{m_j} $ has a singular point $u_3^\ast =z_1 \neq 0$ with the critical value $ - \prod_{j=1}^{q}(-z_j)^{m_j}.$  We shall construct a curve $X(t)$ that satisfies (\ref{(I)}) ,  (\ref{(II)}) of Definition \ref{JK} and $  lim_{t\rightarrow 0} f(X(t)) = -\prod_{j=1}^{q}(-z_j)^{m_j}.$
For $(b_1,b_2,b_3) \in \C^3$ so that $f$ be non-degenerate at infinity, we conclude $$ \{ -\prod_{j=1}^{q}(-z_j)^{m_j} \} \subset {\mathcal K}_\infty(f) \subset \{0 , -\prod_{j=1}^{q}(-z_j)^{m_j} \}$$ thanks to Corollary \ref{cor:inclusionupper }.

First of all, we calculate $\left<\mu_3, \vartheta_u f^W(u)\right>$ and find the facet $\Gamma$ as in Proposition \ref{lemma:C}. For example,
 $\left<\mu_3, \vartheta_u f^W(u)\right> $ has the following form with $U_3 = u_3-z_1:$

$  b_1 {u_1}^2 {u_2}+b_2{u_1}{u_2}^2 (U_3+z_1)^2 +(U_3+z_1)[m_1 U_3^{m_1-1}\prod_{k=2}^{q}(U_3+z_1-z_k)^{m_k}+$ \\
$m_1 U_3^{m_1}(U_3+z_1-z_2)^{m_2-1}\prod_{k=3}^{q}(U_3+z_1-z_k)^{m_k}+...]+3b_3{u_1}^5{u_2}^5{u_3}^3$.\\
We see that Newton polyhedra $\Delta( \left<\mu_1, \vartheta_u f^W(u)\right>), \Delta( \left<\mu_2, \vartheta_u f^W(u)\right>)$ are located in the Newton polyhedron $\Delta( \left<\mu_3, \vartheta_u f^W(u)\right>).$

The facet $\Gamma$ is on the plane containing $(1,2,0),(2,1,0),(0,0,m_1-1)$ and $q = (m_1-1,m_1-1,3),$ i.e. $(q',0)= (m_1-1,m_1-1,0).$

We see that
$\J =\{3\},$ $L_0$ $=$  $max_{i\not = j} \left<(q',0), w_i- w_j\right>$ $=$ $5m_1-5$ and $ \rho=3m_1-3. $
The curve (\ref{(Q)'}) has the expansion\\
$ u_1 = c_1(0) t^{m_1-1} +  c_1(1)t^{m_1} +c_1(2)t^{m_1+1}+ h.o.t. ,$\\
$ u_2 = c_2(0) t^{m_1-1} +  c_2(1)t^{m_1} +  c_2(2)t^{m_1+1}+ h.o.t, $\\
$ u_3 = z_1+ c_3(0) t^{3} +  c_3(1)t^4 +  c_3(2)t^5+ h.o.t. $\\
If we plug these expressions into $\left<\mu_3, \vartheta_u f^W(u)\right>, $ we get an expansion with initial term $t^{3m_1-3}$ ($\left<q, \alpha\right> =3m_1-3$ for $\alpha \in \Gamma$)
$$\{ b_1 (c_1(0))^2 c_2(0)+b_2z_1^2c_1(0)(c_2(0))^2+z_1m_1(c_3(0))^{m_1-1}\prod_{k=2}^{q}(z_1-z_k)^{m_k}  \}t^{3m_1-3}+$$

$\{ b_1[ (c_1(0))^2 c_2(1)+2c_1(1)c_1(0)c_2(0)]+b_2z_1^2[2c_1(0)c_2(1)c_2(0)+c_1(1)(c_2(0))^2] +$\\
$z_1 m_1(m_1-1)c_3(1)(c_3(0))^{m_1-2}\prod_{k=2}^{q}(z_1-z_k)^{m_k}\}  t^{3m_1-2}+$\\
$ \{ b_1[ (c_1(0))^2 c_2(2)+2 c_1(1)c_1(0)c_2(1)+2c_2(0)c_1(2)c_1(0)+(c_1(1))^2c_1(0)c_2(0)]$\\
$+b_2z_1^2[2c_1(0)c_2(2)c_2(0)+2c_1(1)c_2(1)c_2(0)+c_1(2)(c_2(0))^2 +c_1(0)(c_2(1))^2c_2(0) ]$\\
$ + z_1 m_1[\frac{(m_1-1)(m_1-2)}{2}(c_3(1))^2 (c_3(0))^{m_1-3}+
(m_1-1)c_3(2)(c_3(0))^{m_1-2}]\prod_{k=2}^{q}(z_1-z_k)^{m_k}\} t^{3m_1-1}+ h.o.t.$

The coefficient of $t^{3m_1-3}$ depends on $(c_1(0), c_2(0),c_3(0))$ that of $t^{3m_1-2}$ depends on\\
  $(c_1(0), c_2(0),c_3(0), c_1(1),c_2(1),c_3(1))$ that of $t^{3m_1-1}$ depends on \\
   $(c_1(0), c_2(0),c_3(0), c_1(1),c_2(1),c_3(1),c_1(2),c_2(2),c_3(2)).$\\

Thus, we can construct a curve such that $5-5m_1+ ord  \left<\mu_3, \vartheta_u f^W\right>(Q(t)) > 0.$ The minimum parametric length of such a curve $Q(t)$ is $ 2m_1.$
We remark that after the method of  \cite[Theorem 3.5.]{DTT}, the rational curve  with required properties has parametric length $ (1+5\cdot\sum_{j=1}^{q}m_j) (5\cdot\sum_{j=1}^{q}m_j)^2 +1.$

We get the desired curve $X(t)$ as the image of the curve $Q(t)$ by the map

$$  x_1 =u_1^{-1} u_3,  x_2= (u_2 u_3)^{-1},  x_3= u_1^2 u_2^2  u_3.$$

Similarly, as it can be seen in this example for case $2$, the $X(t)$ curve asymptotically approaches the surface $\{ x ; f(x)=0 \}$  as $t \rightarrow0$.

We illustrate below this case with an example $q=2,$ $(z_1,z_2) = (1,2),$
 $ f(x) = ( x^{v_0}-1)^3( x^{v_0}-2 ) + x ^{v_1} +  x^{v_2}+  x^{v_3}.$

For this example, we obtain a curve $X(t) =(x_1(t), x_2(t), x_3(t))$ asymptotically approaching the surface $\{x ; f(x) =0\}$
as follows
$$ x_1(t)= \frac{t^8+\frac{57427 t^7}{20736}+\frac{1645 t^6}{864}+\frac{203 t^5}{144}+\frac{13 t^4}{12}+\frac{t^3}{2}+1}{t^7+t^6+t^5+t^4+t^3+\frac{t^2}{2}}$$
$$x_2(t)=\frac{1}{\left(t^7+t^6+t^5+t^4+t^3+t^2\right) \left(t^8+\frac{57427 t^7}{20736}+\frac{1645 t^6}{864}+\frac{203 t^5}{144}+\frac{13 t^4}{12}+\frac{t^3}{2}+1\right)}$$
{\tiny
$$x_3(t)=\left(t^7+t^6+t^5+t^4+t^3+\frac{t^2}{2}\right)^2 \left(t^7+t^6+t^5+t^4+t^3+t^2\right)^2 \left(t^8+\frac{57427 t^7}{20736}+\frac{1645 t^6}{864}+\frac{203 t^5}{144}+\frac{13
   t^4}{12}+\frac{t^3}{2}+1\right).$$}

\begin{figure}[ht]										
\begin{center}											
\includegraphics[width=8cm]{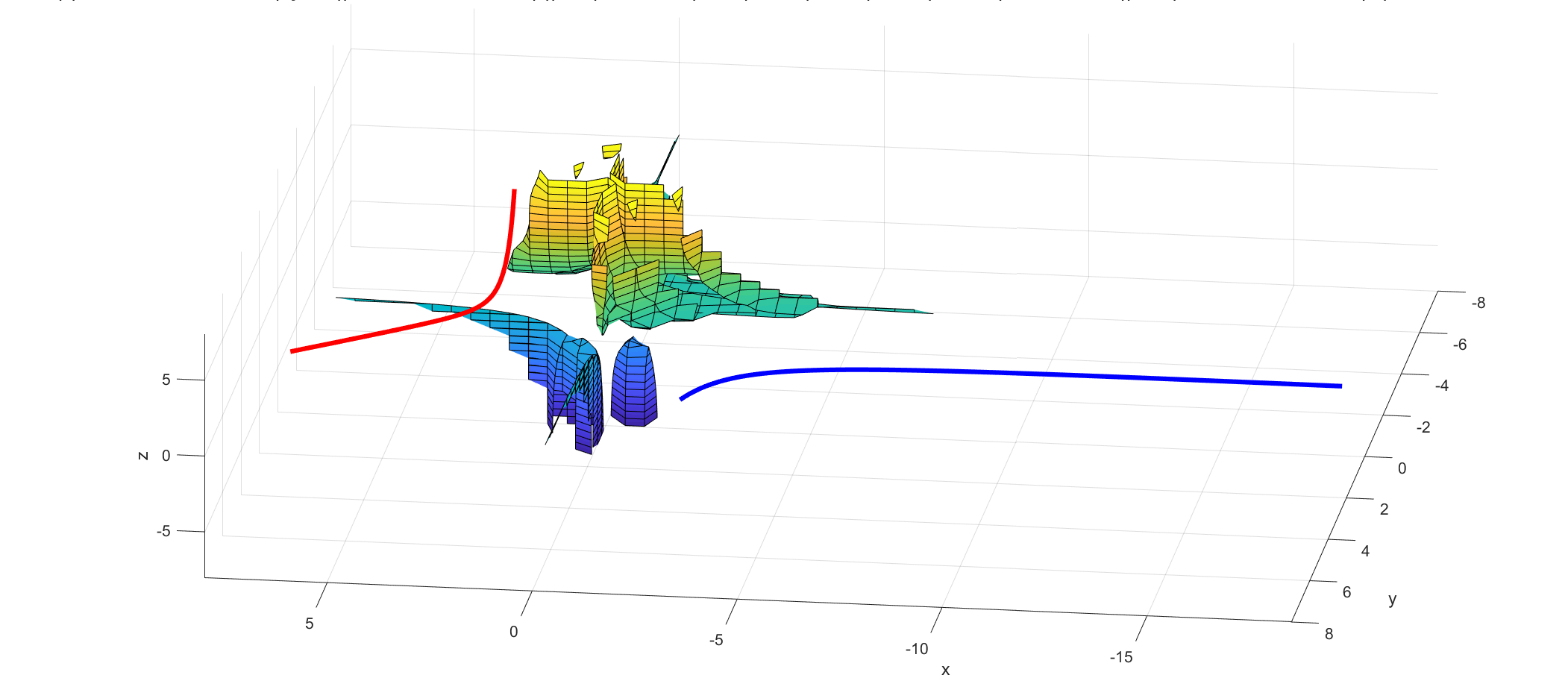} 	
\caption{ Branches of the curve  $X(t)$  }		
\end{center}					
\label{figure532}										
\end{figure}

 On the Figure 5.4,  we see two branches of the curve  that correspond to the asymptotes $ t\rightarrow 0$ from $t>0$ and $t<0$ respectively.

}
\label{example43}
\end{example}


\begin{minipage}[t]{12.2cm}
\begin{flushleft}
{\footnotesize
Susumu TANAB\'E (corresponding author)\\
Department of Mathematics,\\
Galatasaray University,\\
\c{C}{\i}ra$\rm\breve{g}$an cad. 36,\\
Be\c{s}ikta\c{s}, Istanbul, 34357, Turkey.\\
{\it E-mails}:   {tanabe@gsu.edu.tr, tanabesusumu@hotmail.com}\\
ORCID id:  0000-0003-0489-2838  \\

\vspace{0.5cm}
Abuzer G\"UND\"UZ \\
Department of Mathematics,\\
Sakarya University,\\
54050 Sakarya, Turkey.\\
{\it E-mails}:   {abuzergunduz@sakarya.edu.tr, abuzer78@gmail.com}
}\\
ORCID id:   0000-0003-3351-2443 \\
\end{flushleft}
\end{minipage}


\begin{thebibliography}{99}
%
%

\bibitem{CDT}
{\sc  ~Chen Y,  Dias LRG, Takeuchi K,  Tib{\u{a}}r M},
{\em Invertible polynomial maps via Newton non-degeneracy.}  Annales de l'Institut Fourier  {\bf 64} (2014),  1807-1822.

\bibitem{DRT}
{\sc Dias LRG, Ruas MAS, ~Tib\u ar M},
{\em Regularity at infinity of real maps and a {M}orse-{S}ard theorem.}
Journal of  Topology {\bf  5} (2012), 323--340.

\bibitem{DTT}
{\sc Dias LRG, Tanab\'e S, ~Tib\u ar M}, {\em Towards effective detection of the  bifurcation locus of real polynomial maps.}
Foundation of  Computational  Mathematics {\bf 17} (2017), 837-849.

\bibitem{DT}
{\sc  Dias LRG, ~Tib\u ar M},
{\em Detecting bifurcation values at infinity of real polynomials.}   Math. Z.    {\bf 279} (2015), 311--319.

\bibitem{Fulton} {\sc  Fulton W}, {\em   Introduction to Toric Varieties,} Annals of mathematics studies {\bf 131}, Princeton University Press, Princeton, NJ, 1993.

\bibitem{GZKDiscriminant}{\sc I.M.Gel'fand, A.V.Zelevinski{\u i}, M.M.Kapranov},
 {\it Discriminants of polynomials in several variables and triangulations of Newton polyhedra,} Leningrad Math. J. {\bf 2} (1991), 449-505.


\bibitem{Ishikawa} {\sc  Ishikawa M}, {\em The bifurcation set of a complex polynomial function of two variables and the Newton polygons of singularities at infinity.} J. Math. Soc. Japan   {\bf 54}  (2002),  161--196.

\bibitem{JK0}{\sc  Jelonek Z,  Kurdyka K}, {\em On asymptotic critical values of a complex polynomial.} J. Reine  Angew.  Math.  {\bf 565} (2003), 1--11.

\bibitem{JK}  {\sc Jelonek Z,  Kurdyka K}, {\em Reaching generalized critical values of a polynomial.} Math. Z. {\bf 276} (2014), 557--570.

\bibitem{NZ}
{\sc N\'emethi A,  Zaharia A},  {\em On the bifurcation
set of a polynomial function and Newton boundary.}  Publ. Res. Inst. Math. Sci. {\bf 26} (1990), 681--689.

\bibitem{Oka} {\sc  Oka M}, {\em    Non-degenerate complete intersection singularity,}Hermann, Paris,  1997.

\bibitem{Parus}  {\sc Parusi\'nski A}, {\em On the bifurcation set of complex polynomial with isolated singularities at infinity.} Compositio Math. {\bf 97} (1995), 369--384.

\bibitem{Ta}{\sc  Takeuchi K}, {\em Bifurcation values of polynomial functions and perverse sheaves.} Annales de l'Institut Fourier {\bf 70} (2020),  597-619.

\bibitem{TGE}{\sc  Tanab\'e S,  G\"und\"uz A,
 Ersoy B A}, {\em On real curve construction for asymptotic critical value set of a polynomial map.} Comptes rendus de l'Acad\'emie bulgare des Sciences  (2021), to appear.

\bibitem{Z96} {\sc Zaharia A}, {\em On the bifurcation set of a polynomial function and Newton boundary, II.}  Kodai Math. J.    {\bf 19} (1996), 218--233.
\end{thebibliography}
\end{document}